\documentclass[12pt]{article} 
\title{Adjusted Scores for Discrete Langevin Algorithms}
\usepackage{times}

\usepackage[english]{babel}
\usepackage{amsthm}
\usepackage{amsmath}
\usepackage{amsfonts}
\usepackage{amssymb}
\usepackage{makeidx}
\usepackage{graphicx}
\usepackage{subfigure}
\usepackage[margin=2cm]{geometry}
\usepackage[hidelinks]{hyperref}
\usepackage{xcolor}
\usepackage[justification=centering]{caption}
\usepackage{subcaption}
\usepackage{algorithm}
\usepackage{algorithmic}
\usepackage{cleveref}
\crefname{equation}{Eq.\ }{Eqs.\ }
\usepackage{balance}
\usepackage{mathtools}

\usepackage{tikz}
\usepackage{natbib}

\newtheorem{theorem}{Theorem}
\newtheorem{proposition}{Proposition}

\newtheorem{assumptionA}{{A\!}}
\crefformat{assumptionA}{{\bf A}{\normalfont #2#1#3}}

\crefformat{assumptionS}{{\bf S}{\normalfont #2#1#3}}

\newtheorem{assumptionT}{{T\!}}
\crefformat{assumptionT}{{\bf T}{\normalfont #2#1#3}}

\newcommand{\X}{\mathcal X}
\newcommand{\1}{\mathbf 1}

\newcommand{\bE}{\mathbb E}
\newcommand{\bP}{\mathbb P}

\usepackage{authblk}
\author[1]{Armand Gissler}
\author[2]{Saeed Saremi}
\author[1]{Francis Bach}
\affil[1]{Inria, Ecole Normale Sup\'erieure, Universit\'e PSL.}
\affil[2]{Frontier Research, Prescient Design, Genentech.}

\begin{document}

\maketitle

\begin{abstract}%
  Sampling from discrete distributions is a ubiquitous task in machine learning, recently revisited by the emergence of discrete diffusion models. 
  While Langevin algorithms constitute the state of the art for continuous spaces, discrete versions lack similar theoretical guarantees when the step-size becomes small.
  In this paper, we address this limitation by interpreting discrete sampling algorithms as discretizations of continuous-time dynamics on the hypercube.
  In particular, we describe several score functions for discrete algorithms which result in approximations of Glauber dynamics for the correct target distribution.
  We also compute upper bounds for the contraction of these algorithms, with or without Metropolis adjustment.
\end{abstract}

\section{Introduction}

Drawing samples from a target distribution on finite state spaces is a fundamental problem in statistics and machine learning, with applications in various domains 
such as combinatorial optimization~\citep{sun2023revisiting}, molecular generation~\citep{luo2021graphdf}, and text generation~\citep{holtzman2019curious}.
Recently, the generalization of generative diffusion-based or denoising-based models from Euclidean spaces to discrete state spaces has revived interest in sampling in discrete spaces~\citep{austin2021structured,hoogeboom2021argmax,campbell2022continuous,sun2023score,pham2025discrete,bach2025sampling}.
In this paper, we focus on two algorithms for this problem: (1) discrete Langevin algorithms proposed by \cite{zhang2022langevinlike} directly inspired by the unadjusted Langevin algorithm \citep[ULA,][]{durmus2017nonasymptotic}, and the Metropolis-adjusted Langevin algorithm, \citep[MALA,][]{roberts1996exponential}, and (2) a two-stage discrete proximal algorithms designed by \cite{bach2025sampling}, following an idea of \cite{lee2021structured} on continuous state spaces.

Contrary to Langevin algorithms for sampling in $\mathbb R^d$ \citep[see, e.g.,][]{roberts1996exponential,chewi2023logconcave}, these discrete algorithms were not formulated by \cite{zhang2022langevinlike} and \cite{bach2025sampling} by the discretization in time of a continuous-time dynamics which converges to the target distribution.
This is the main motivation of this paper, as we tackle the following question:
\begin{quote}  
    \emph{Can we frame (improved versions of) the discrete Langevin algorithms as discretizations of underlying continuous-time dynamics and  obtain improved convergence rates?}
\end{quote}

\paragraph{Contributions.}
Overall, we study four algorithms described in \Cref{sec:unadjusted,sec:MH}: the DULA (discrete ULA) algorithm and its Metropolis-adjusted counter part, the DMALA algorithm from \citet{zhang2022langevinlike}, the DUPS algorithm (discrete unadjusted proximal sampler) from \citet{bach2025sampling} and the DMAPS (discrete Metropolis-adjusted proximal sampler) that we propose in this paper defined as the DUPS algorithm with a Metropolis adjustment. We focus on the binary hypercube $\X = \{-1,1\}^d$, noting that most of our developments extend to more general product spaces. 
Throughout the paper, we denote $p\colon\X\to(0,1)$ the target probability distribution. 

All these algorithms are ``score-based'': in order to mimic the Langevin algorithm, they exploit a score function $s:\{-1,1\}^d \to \mathbb R^d$ which we refer to as \emph{Stein score}, defined as the gradient of a continuation of $\log p(\cdot)$ on $\mathbb R^d$.
This is in part motivated by the settings of discrete diffusion models, in which we only have access to the target distribution via its score, which is learned from data.
However, the definition of $s(\cdot)$ remains ambiguous as it depends on the chosen continuation of $\log p(\cdot)$.
Our first main contribution is to address this ambiguity. We introduce two novel score functions, which we name \emph{Gibbs score} and \emph{Glauber score}, with improved theoretical properties: crucially, the definitions of these score functions are unambiguous, and the resulting discrete Langevin algorithms have well-defined limiting behavior as the step-size goes to zero. 

Besides, in \Cref{sec:Gibbs}, we identify the Glauber dynamics~\citep{glauber1963time} as a continuous-time process that admits the target distribution as invariant distribution.
In \Cref{sec:unadjusted}, we prove that both the DULA and DUPS algorithms can be derived from the Glauber dynamics, which is the continuous-time counterpart to Gibbs sampling~\citep{geman1984stochastic}.
Indeed, for the Gibbs score, the DULA algorithm approximates Gibbs sampling in the small step-size regime.
Likewise, the DUPS algorithm with Glauber score behaves like a Gibbs iteration on a proximal distribution for small step-sizes, but can outperform it for larger step-size sizes.

In \Cref{sec:unadjusted}, we improve the convergence analysis from \cite{bach2025sampling} for DULA and DUPS when we choose the adjusted Gibbs score: we obtain contraction for small step-size independently of the target distribution, and we find that the error between their invariant probability distributions and the target distribution tends to zero when the step-size vanishes.

Furthermore,
we analyze in \Cref{sec:MH} the effect of a Metropolis adjustment~\citep{metropolis1953equation,hastings1970monte} on these algorithms, resulting in the DMALA and the DMAPS algorithm.
We find conditions for contraction from the contraction rate of the unadjusted algorithms and properties of the acceptance rate.

Finally in \Cref{sec:experiments-dula-DUPS}, we perform experiments to support our theoretical results on multi-modal distributions:
a model of mixture of independent bits, and two models for the magnetization, the Ising model~\citep{ising1925beitrag} and the Curie-Weiss model~\citep{curie1895proprietes,weiss1907hypothese}, where we show how Gibbs sampling is outperformed by the new discrete Langevin algorithms.

\section{Glauber dynamics and convergence of Gibbs sampling}\label{sec:Gibbs}
Consider the finite state-space $\X=\{-1,1\}^d$ on which we want to produce samples from the probability distribution $p(\cdot)$ defined on $\X$, that we call target distribution.
Gibbs sampling \citep[GS,][]{geman1984stochastic} consists in producing a Markov chain $\{x_n\}_{n\in\mathbb N}$ such that the distribution of $x_n$ converges to $p(\cdot)$ when $n$ tends to infinity. More precisely, from iteration $n\in\mathbb N$, given $x_n\in\X$, it chooses a coordinate uniformly randomly $i\in\{1,\dots,d\}$ and sample $[x_{n+1}]_i$ according to the $i$-th marginal distribution of $p(\cdot)$, given that $[x_{n+1}]_{-i}=[x_{n}]_{-i}$.

For the sake of our analysis, we suppose here that the distribution $p$ is strictly positive on $\X$.
The Glauber dynamics~\citep{glauber1963time} on $\X$ results in a continuous-time jump Markov process $\{\mathbf X_t\}_{t\geqslant0}$ on $\X$ which converges to the target distribution $p(\cdot)$.
In order to describe this dynamics, we consider $\mathbf Q\in\mathbb R^{\X\times\X}$ its generator matrix \citep[see][for an introduction to jump processes]{norris1998markov}. 
To that end, we define the \emph{Glauber score} $\delta \log p\colon\{-1,1\}^d \to\mathbb R^d$ by
\begin{equation}\label{eq:natural-score}
    [\delta \log p (x)]_i = \frac12 \log p (1,x_{-i})  -  \frac12 \log p (-1,x_{-i}) \qquad \text{ for every } x\in\X \text{ and } i\in\{1,\dots,d\},
\end{equation}
where we use the notation $(u,x_{-i})$ to denote the vector in $\{-1,1\}^d$ where $x_i=u$ and the other components of $x$ are kept unchanged. Note that the Glauber score is a form of discretized gradient, and that it equals the Stein score $s(\cdot)=\nabla\log p(\cdot)$, when the potential $\log p(\cdot)$ is linear (in which case $p(\cdot)$ is the distribution of $d$ independent bits) or quadratic with zero diagonal elements of its Hessian matrix (like for the Ising model).
Otherwise, the error between Glauber and Stein scores is proportional to
\begin{equation}
    \beta_2 = \max_{x\neq y}  \frac{\|\delta \log p (x) - \delta \log p (y)  \|_\infty }{\| x-y  \|_1  } .
\end{equation}
The generator matrix associated to Glauber dynamics is defined by its nondiagonal elements
\begin{equation}\label{eq:Glauber-dynamics}
    \mathbf Q_{xy} = \sum_{i=1}^d \sigma \left( -2x_i \delta \log p(x)_i \right)\mathbf 1\{y = (-x_i,x_{-i})\} .
\end{equation}
By definition of the generator matrix, between times $t$ and $t+dt$, independently (for each $i$) with probability $\sigma \left( -2x_i \delta \log p(x)_i \right) dt$, the $i$-th component is flipped, and with probability $ 1 - dt \sum_{i=1}^d  \sigma \left( -2x_i \delta \log p(x)_i \right)$, no move is performed. This leads to the natural time-discretization with step-size $h>0$, where to go from time $(k-1)h$ to time $kh$, we use the transition kernel $t(x'|x)=h \mathbf Q_{xx'}\mathbf 1\{x\neq x'\} + \left(1 - h \sum_{y \neq x} \mathbf Q_{xy} \right)\mathbf{1}\{x=x'\}$.
It is known that when $h$ tends to zero, the discrete-time process converges to the continuous-time process~\citep[Theorem 2.8.2(b) from][]{norris1998markov}. 
In our results below, we will show that the DULA and DUPS algorithms with the properly chosen score are such discretizations.

Then, 
it is natural to define an algorithm obtained by the time-discretization of the Glauber dynamics.
This results in damped\footnote{We refer to it as ``damped'', because, for $e^{-2/\eta}=1/d$, this is exactly Gibbs sampling with a random choice of site to update, while for $e^{-2/\eta}<1/d$, only with probability $de^{-2/\eta}$, a Gibbs sampling step is performed (and otherwise no move is performed).} Gibbs sampling, with step-size $h=e^{-2/\eta}\leqslant1/d$ (to ensure the positivity of the probability of the Markov chain to remain at the same state).
It can be described by its transition kernel $t_{\rm GS}(\cdot|\cdot)$, which  writes as, for $x,x'\in\X$:
\begin{equation}\label{eq:transition-gibbs}
    t_{\rm GS}(x'|x)  = \left\{ 
    \begin{array}{ll}
        \left( 1 -  e^{-2/\eta} \sum_{i=1}^d  \sigma( - 2x_i \delta\log p (x)_i) \right)  &  \text{if } x'=x  \\
         e^{-2/\eta} \sigma( - 2x_i \delta\log p (x)_i) &  \text{if } x'=(-x_i,x_{-i}) .
    \end{array}
    \right.
\end{equation}
By positivity of $p(\cdot)$, this defines an irreducible transition kernel (every state is reachable with positive probability in $d$ iterations). Besides, it satisfies the property of detailed balance with respect to the target distribution $p(\cdot)$, which implies the convergence to $p(\cdot)$.

We recover below a traditional upper bound for the convergence rate of Gibbs sampling using a coupling method. To that end, we define the Wasserstein distance $\mathcal W(\cdot,\cdot)$ between two probability distributions $p,q$ on $\X$ as
\begin{equation}
    \mathcal{W}(p,q) = \min_{x\sim p, y \sim q} \bE [\ell (x,y)],
\end{equation}
where we use the Hamming loss $\ell(x,y)=\sum_{i=1}^d \mathbf 1\{x_i\neq y_i\}$.
Alternatively, the performance of Gibbs sampling is often expressed in terms of mixing time: if the Dobrushin's condition is satisfied, i.e., if the maximum eigenvalue of the influence matrix (composed of the TV distance of marginal distributions sampled from the target distribution) is less than $1-\varepsilon$, then the mixing time of Gibbs sampling is of order $O(d\log d)$, see the works of \cite{dobrushin1970prescribing,dobrushin1985constructive,weitz2005combinatorial,hayes2006simple} and their references therein.

We give the contraction rate of Gibbs sampling in the following theorem, where the constraint $d\beta_2\leqslant1$ essentially corresponds to Dobrushin's condition (see proof in \Cref{app:gibbs}).

\begin{theorem}[Contraction of Gibbs sampling]\label{t:gibbs}
    We have the following contraction property when $d\beta_2\leqslant1$ and $e^{-2/\eta}\leqslant1/d$:
    \begin{equation}
        \mathcal W (  t_{\rm GS}(\cdot|x),t_{\rm GS}(\cdot|y) )
        \leqslant
        \left( 
            1 -  e^{-2/\eta}(1-d\beta_2)
        \right)  \ell(x,y) \qquad \text{ for every } x,y\in\X.
    \end{equation}  
\end{theorem}
A direct consequence is the following convergence property: for any distribution $q(\cdot)$ on $\X$,
\begin{equation}
    \mathcal W ( q\, t^n(\cdot) , p(\cdot) ) \leqslant \left( 1 -  e^{-2/\eta}(1-d\beta_2) \right)^n \mathcal W(q(\cdot),p(\cdot)). 
\end{equation}

\section{Convergence of discrete unadjusted samplers}\label{sec:unadjusted}
 
\subsection{Contraction and approximation error of DULA}\label{sec:DULA}

\cite{zhang2022langevinlike} proposed Langevin-type samplers for discrete state-spaces, inspired by ULA (unadjusted Langevin algorithm, \cite{durmus2017nonasymptotic}) and MALA (Metropolis-adjusted Langevin algorithm, \cite{roberts1996exponential}) for continuous state-spaces. 
 DULA (discrete ULA) proceeds to define a Markov chain $\{x_n\}_{n\in\mathbb N}$ whose stationary distribution is expected to be close to the target distribution.
More precisely, consider a target distribution $p(\cdot)$ positive on a finite state-space $\X\subset \mathbb R^d$ (that we suppose to be $\{-1,1\}^d$ for simplicity). Then, DULA consists in considering the following transition kernel on $\X$ \citep[which corresponds to the usual continuous Gaussian Langevin step restricted to $\{-1,1\}^d$, see][]{zhang2022langevinlike}:
\begin{equation}
    t_{\rm DULA} (x'|x) \propto \exp\left(  \left(  \frac{x}{\eta} +\frac{s(x)}{2}  \right)^\top x'  \right) ,
\end{equation}
where $\eta>0$ is a step-size parameter, and
the function $s\colon\X\to\mathbb R^d$ is called the score function and is supposed known. 
For continuous Langevin algorithms, the Stein score function is the standard score and writes as $s(\cdot)=\nabla\log p(\cdot)$. 
However, in the absence of a gradient for functions on discrete spaces, the choice of $s(\cdot)$ boils down to the interpolation $\tilde p(\cdot)$ on $\mathbb R^d$ of the target distribution $p(\cdot)$ in order to define the Stein score as $s(\cdot)=\nabla \log \tilde p(\cdot)$. 
Then, one iteration of DULA at a state $x\in\X$ consists in flipping each coordinate $i=1,\dots,d$ of $x$ independently with probability $\sigma(x_i/\eta+s(x)_i/2)$.
This is intended to work as a parallelized Gibbs sampler.

A natural score then derives  from the finite difference approximation $s(\cdot)=\delta\log p(\cdot)$, defined in \eqref{eq:natural-score}.
As underlined later, this score allows the unadjusted algorithms to derive from a Glauber dynamics, hence we refer to this score as the Glauber score.
\cite{bach2025sampling} proved that DULA is contractant under conditions on the score and on the constants
\begin{align}
    \beta_1 & = \max_{x\in\X} \| s(x)\|_\infty \label{eq:beta1} \\
    \beta_2 & = \max_{x\neq y} \cfrac{\|s(x)-s(y)\|_\infty}{\|x-y\|_1}. \label{eq:beta2}
\end{align}
When $2d\beta_2 \leqslant e^{-\beta_1}$, \citet[Proposition 3.1]{bach2025sampling} obtain the following contraction property~:
    \begin{equation}\label{eq:contraction-dula}
        \mathcal{W} \left( t_{\rm DULA}(\cdot|x) , t_{\rm DULA}(\cdot|y) \right) \leqslant \left( 1-\frac12 e^{-\frac2\eta-\beta_1}  \right) \ell (x,y) .
    \end{equation}
Note that the contraction rate could be improved to $1-\frac{1}{2}e^{-1/\eta}$, assuming that $\eta\leqslant1/\beta_1$, in a similar fashion than \Cref{t:contraction-dula-etasmall}.
Furthermore, \cite{bach2025sampling} obtained the following approximation error:
\begin{equation}\label{eq:error-DULA}
        \mathcal W \left( \hat p_{\rm DULA}(\cdot) , p(\cdot) \right)
        \leqslant 2d \left( 2d \beta_1 e^{2\beta_1} + \sqrt{d \beta_1 e^{2\beta_1} } \right) .
    \end{equation}
The issue with the bound \eqref{eq:error-DULA} is that it does not show the usual property of unadjusted Langevin algorithms which is that the approximation error tends to $0$ when the step-size tends to $0$. In order to address it, we propose below to write the DULA algorithm as the discretization of the Glauber dynamics given by the generator matrix \eqref{eq:Glauber-dynamics}, so that  DULA behaves like the Gibbs sampling when $\eta$ is small.
See \Cref{app:discretization-dula} for the proof of \Cref{t:discretization-dula}.

\begin{theorem}\label{t:discretization-dula}
Consider the Markov process $\{\mathbf X_t\}_{t\geqslant0}$ defined via its generator matrix $\mathbf Q$, such that, for $x\neq y \in\X$:
\begin{equation}\label{eq:generator-DULA}
    \mathbf Q_{xy} = \sum_{i=1}^d \exp (-x_i s(x)_i)\mathbf 1\{ y = [-x_i,x_{-i}] \} .
\end{equation}
Then, the DULA can be seen as a discretization of $\{\mathbf X_t\}_{t\geqslant0}$ with step-size $e^{-2/\eta}$.
Besides, the target distribution $p(\cdot)$ is the invariant probability measure of $p(\cdot)$ with the Gibbs score, defined by
\begin{equation}\label{eq:gibbs-score}
    s(x)_i = x_i  \log\left( 1+ \exp\left( 2x_i\delta\log p(x)_i \right) \right) ,
\end{equation}
in which case the generator matrix \eqref{eq:generator-DULA} equals the one of the Glauber dynamics \eqref{eq:Glauber-dynamics}. 
\end{theorem}

We  refer to the score \eqref{eq:gibbs-score} as the Gibbs score.
We  improve the bound on the contraction rate of DULA with Gibbs score when the step-size is small in \Cref{t:contraction-dula-etasmall} (see proof in \Cref{app:contraction-dula-etasmall}).

\begin{theorem}\label{t:contraction-dula-etasmall}
    Let $\beta_2\geqslant0$ be defined by  \eqref{eq:beta2}.
    Consider the Gibbs score $s(\cdot)$  by \eqref{eq:gibbs-score}.
    If $4d\beta_2\leqslant1$, then we have the following contraction for the DULA :
    \begin{equation}
        \mathcal{W} \left( t_{\rm DULA}(\cdot|x) , t_{\rm DULA}(\cdot|y) \right) \leqslant \left( 1-\frac14 e^{-\frac2\eta}  \right) \ell (x,y) .
    \end{equation}
\end{theorem}

Moreover, we obtain the following error result for the stationary distribution (see the proof of \Cref{t:error-DMALA-eta-small} in \Cref{app:error-DMALA-eta-small}).

\begin{theorem}\label{t:error-DMALA-eta-small}
Consider the Gibbs score \eqref{eq:gibbs-score}.
    Assume $4d\beta_2\leqslant1$ and $de^{-2/\eta}\leqslant1$. Then, the invariant measure $\hat p_{\rm DULA}(\cdot)$ of $t_{\rm DULA}(\cdot|\cdot)$ satisfies
    \begin{equation}
        \mathcal W(\hat p_{\rm DULA}(\cdot), p(\cdot) ) \leqslant \frac{4d}{1+ e^ {2/\eta}}.
    \end{equation}
\end{theorem}

Note that the contraction rate obtained for DULA with Gibbs score does not improve from Gibbs sampling:
under the assumption for \Cref{t:contraction-dula-etasmall} that $4d\beta_2\leqslant1$, the upper bound of the contraction rate for Gibbs sampling obtained in \Cref{t:gibbs} is smaller than the one we have for DULA, the difference being at least at least $e^{-2/\eta}/2$, while the invariant probability measure of DULA only approximates the target distribution, contrary to Gibbs sampling which is exact.
However, the empirical mixing times in \Cref{sec:experiments-dula-DUPS} suggest that DULA is often faster than Gibbs sampling for moderate step-sizes while the approximation error remains low.

    \subsection{Convergence of a proximal sampler}\label{sec:DUPS}

    \cite{bach2025sampling} introduced a two-stage proximal sampler, inspired from the work of \cite{lee2021structured}, that we call from now on the discrete unadjusted proximal sampler (DUPS).  
    Consider a Gibbs sampling iteration for the joint ``proximal'' distribution $p_{\rm prox} (x,z) \propto p(x) \exp(x^\top z)$, given by the transition kernel $t_{\rm prox}(\cdot|\cdot)$, defined by
    \begin{equation}
        t_{\rm prox}(x'|x) = \sum_{z\in\{-1,1\}^d} u(z|x) v(x'|z),
    \end{equation}
    where
    \begin{equation}
        u(z|x) \propto \exp\left( \frac{x^\top z}{\eta} \right) \qquad;\qquad v(x'|z) \propto  p(x')\exp\left( \frac{z^\top x'}{\eta} \right),
    \end{equation}
    for a fixed step-size $\eta>0$.
    By reversibility of the Gibbs sampler, the transition kernel $t_{\rm prox}(\cdot|\cdot)$ is reversible with respect to the target distribution $p(\cdot)$.
    The DUPS is a score-based approximation of this proximal Gibbs sampler. Consider a score function $s(\cdot)$ and the transition kernel $t_{\rm DUPS}(\cdot|\cdot)$, defined by
    \begin{equation}
        t_{\rm DUPS}(x'|x) = \sum_{z\in\{-1,1\}^d} u(z|x) \hat v(x'|z)
    \end{equation}
    with
    \begin{equation}
        \hat v (x'|z) \propto \exp\left(  \left( \frac{z}{\eta} + s(z) \right)^\top x' \right) .
    \end{equation}
    In general, $t_{\rm DUPS}(\cdot|\cdot)$ does not admit $p(\cdot)$ as invariant probability distribution, but only when $s(z)^\top x' = \log p(x')$, i.e., when there exists a fixed $s=s(z)\in\mathbb R^d$ such that $p(x)\propto \exp(s^\top x)$, that is the target distribution is the distribution of independent bits.
    When $4d\beta_2e^{4\beta_1}\leqslant1$, we improve in the following equation (proven in \Cref{app:contraction-dups}) the contractivity result from \cite{bach2025sampling},
    \begin{equation}\label{t:contraction-dups}
        \mathcal W \left( t_{\rm DUPS} (\cdot|x) , t_{\rm DUPS} (\cdot|y)   \right)
        \leqslant
        \left( 1 - 2\sigma(-2/\eta) \right)  \ell (x,y) .
    \end{equation}
Besides, we recall the error to the target distribution proven by \cite{bach2025sampling}.
When $4d\beta_2\leqslant e^{-4\beta_1}$ and $\exp(-2/\eta-2\beta_1)\leqslant 1/d$, then
\begin{equation}
        \mathcal W \left( p_{\rm DUPS}(\cdot), p \right) \leqslant 12 d \sqrt{\beta_2 d}.
\end{equation}
This bound fails to obtain an error which tends to $0$ when the step-size tends to $0$. We find then that the Glauber score $s(\cdot)=\delta \log p(\cdot)$ defined in  \eqref{eq:natural-score} makes the DUPS a discretization of a continuous-time Markov process which admits $p(\cdot)$ as invariant probability measure.

This is underlined in the following theorem, 
see proof in \Cref{app:discretization-DUPS}.
\begin{theorem}\label{t:discretization-DUPS}
    Consider the Markov process $\{\mathbf X_t\}_{t\geqslant0}$ defined via its generator matrix $\mathbf Q^{\rm DUPS}$, such that, for $x\neq y \in\X$:
\begin{equation}\label{eq:generator-DUPS}
    \mathbf Q^{\rm DUPS}_{xy} =  \sum_{i=1}^d \left(1+\exp (-2x_i s(x)_i)\right)\mathbf 1\{ y = [-x_i,x_{-i}] \} .
\end{equation}
Then, the DUPS can be seen as a discretization of $\{\mathbf X_t\}_{t\geqslant0}$ with step-size $e^{-2/\eta}$.
Besides, the target distribution $p(\cdot)$ is the invariant probability measure of $p(\cdot)$ when 
$s(\cdot)=\delta\log p(\cdot)$ is the Glauber score,
in which case the DUPS behaves like the Gibbs-proximal sampling $t_{\rm prox}(\cdot|\cdot)$ when $\eta$ is small. 
\end{theorem}

We deliver now a similar analysis than for the DULA. 
Yet, the Glauber score does not satisfy the inequality $x_i s(x)_i\geqslant0$, that we used for the DULA with Gibbs score to obtain better contraction rates for small step-sizes $\eta>0$.
Instead, we assume $x_i s(x)_i\geqslant -\frac{1}{2\eta}$, which holds for Glauber score when $\eta>0$ is small. The proof of the following theorem is given in \Cref{app:contraction-DUPS-eta-small}.
\begin{theorem}\label{t:contraction-DUPS-eta-small}
    Suppose that $\eta>0$ is sufficiently small so that $x_i s(x)_i\geqslant -\frac{1}{2\eta}$. If $8d\beta_2\leqslant1$, then,
    \begin{equation}
        \mathcal W\left(  t_{\rm DUPS} (\cdot|x)  ,  t_{\rm DUPS} (\cdot|y)  \right)
        \leqslant
        \left(1 - \frac{1}{2} e^{-1/\eta}\right)  \ell (x,y) .
    \end{equation}
\end{theorem}
In the following theorem, we are finally able to obtain an error of the DUPS that tends to $0$ for vanishing step-sizes, when using the Glauber score.
The proof of \Cref{t:error-dups-eta-small} is delayed to \Cref{app:error-dups-eta-small}.

\begin{theorem}\label{t:error-dups-eta-small}
    Suppose that $s(\cdot)=\delta \log p(\cdot)$ is the Glauber score, $x_i s(x)_i \geqslant-\frac{1}{2\eta}$, and $8d\beta_2\leqslant1$.
    Then, $t_{\rm DUPS}(\cdot|\cdot)$ converges to its invariant probability measure $\hat p_{\rm DUPS}(\cdot)$, with
    \begin{equation}
        \mathcal{W} \left( \hat p_{\rm DUPS}(\cdot) , p(\cdot) \right) \leqslant \frac{d^3}{2} (1+e^{-1/\eta})^{d-1} \cdot e^{-1/\eta} .
    \end{equation}
\end{theorem}
All in all, we obtain a better upper bound for the contraction rate of DUPS than for DMALA when the step-size is small when choosing the right score ( $1-\frac{1}{2} e^{-1/\eta}$ instead of $1-\frac{1}{4} e^{-2/\eta}$). 
However, it seems that DULA makes a better approximation of the target distribution (proportional to $e^{-1/\eta}$ instead of $e^{-2/\eta}$).
These observations are supported by our experiments presented in \Cref{sec:experiments-dula-DUPS}.

\section{Effect of a Metropolis adjustment on discrete Langevin algorithms}\label{sec:MH}

The main issue with the DULA and the DUPS presented in the previous section is the error between the invariant probability measures of their transition kernels and the target distribution $p(\cdot)$.
A standard method to correct the invariant probability measure is to add a Metropolis acceptance step~\citep{metropolis1953equation}. This has been studied for Langevin algorithms in the continuous setting~\citep{roberts1996exponential}, it is therefore natural to analyze Metropolis-adjusted discrete Langevin algorithms.
In this section, we consider a general two-stage\footnote{This allows to apply the results below to both the DULA and the DUPS kernels.} transition kernel $t(\cdot|\cdot)$ on $\X$ such that
\begin{equation}
    t(x'|x) = \sum_{z\in\X} u(z|x) v(x'|z) ,
\end{equation}
given two transition kernels $u(\cdot|\cdot)$ and $v(\cdot|\cdot)$ on $\X$.
Therefore, a standard Metropolis approach for the two-stage kernel $t(\cdot|\cdot)$ would imply to define the following acceptance rate:
$$
A_{\rm MH}(x'|x) = \min\left\{ 1 , \frac{p(x')}{p(x)}\frac{t(x|x')}{t(x'|x)} \right\} = \min\left\{ 1 , \frac{p(x')}{p(x)}\frac{\sum_{\bar z} u(\bar z|x') v(x|\bar z)}{\sum_{\bar z} u(\bar z|x) v(x'|\bar z) } \right\} ,
$$
which would be hard to compute in practice when the dimension is high due to the sum in the acceptance rate.
We consider instead the acceptance rate
\begin{equation}
\label{eq:Azxx'}
A_z (x'|x) = \min\left\{ 1 , \frac{p(x')}{p(x)} \frac{u(z|x')v(x|z)}{u(z|x)v(x'|z)}  \right\} .
\end{equation}

The corresponding algorithm summarizes in the following lines. Given produced samples $x_0,\dots,x_t$ in $\X$, we construct $x_{t+1}\in\X$ as follows:
\begin{enumerate}
    \item[(i)] sample $z_{t+1}\sim u(\cdot | x_t)$,
    \item [(ii)] sample $x_{t+1}'\sim v(\cdot|z_{t+1})$,
    \item[(iii)] with probability $A_{z_{t+1}}(x_t,x'_{t+1})$, set $x_{t+1}=x'_{t+1}$, otherwise $x_{t+1}=x_t$.
\end{enumerate}
We define therefore the associated transition kernel on $\X$:
\begin{equation}
    t_{\rm MH} (x'|x) = \sum_{z\in\X} u(z|x) v(x'|z)A_z (x'|x)
    +  \mathbf1\{x'=x \} \sum_{\bar x'\in\X}\sum_{\bar z\in\X}u(\bar z|x) v(\bar x'|\bar z)
    \left(1-A_{\bar z} (\bar x'|x)\right) .
\end{equation}

\subsection{Convergence to the invariant probability distribution}

Under mild assumptions, the transition kernel $t_{\rm MH}(\cdot|\cdot)$ admits $p(\cdot)$ as a unique invariant probability distribution, i.e.,
$$
p t_{\rm MH} (x') \coloneqq \sum_{x\in\X} p(x) t_{\rm MH}(x'|x) = p(x') .
$$
In fact, since the state space $\X$ is finite, obtaining this is easily achieved.
Indeed, if for every $x,x'\in\X$, there exists $z\in\X$ such that $u(z|x)$, $v(x'|z)$ and $A_z(x'|x)$ are positive, then $t_{\rm MH}(\cdot|\cdot)$ is irreducible and therefore positive recurrent, i.e., it converges to its unique invariant probability measure $\pi(\cdot)$

Using the detailed balance property of $t_{\rm MH}$ with respect to $p(\cdot)$, we find that $p(\cdot)$ corresponds to the invariant probability measure of $t_{\rm MH}$.
To prove this, we observe that for $x,x' \in \X$, if $x\neq x'$, we have: 
    \begin{align*}
        t_{\rm MH} (x' | x) p(x) & = \sum_{z\in\X} u(z|x) v(x'| z) A_z(x'|x) p(x)  \\
        & = \sum_{z\in\X}\min\left\{ p(x)u(z|x)v(x'|z) , p(x')u(z|x')v(x|z)  \right\} = t_{\rm MH}(x|x') p (x') .
    \end{align*}
We deduce the following convergence result.
\begin{theorem}\label{th:convergence-to-target}
    If, for every $x,x'\in \X$, there exists $z\in\X$ such that $u(z|x)$, $v(x'|z)$ and $A_z(x'|x)$ are positive, then,  for every $x\in\X$,
    $
        \lim_{n\to+\infty} t_{\rm MH}^n(\cdot | x ) = p(\cdot).
    $
\end{theorem}

\subsection{Contraction Property}

\Cref{th:convergence-to-target} provides sufficient conditions for the convergence of $t_{\rm MH}(\cdot|\cdot)$ to the target distribution~$p(\cdot)$. In order to quantify the convergence of the kernel $t_{\rm MH}(\cdot|\cdot)$, we 
consider the contraction assumption of the kernel $t(\cdot|\cdot)$. 
\begin{assumptionT}\label{T:contraction}
    There exists $\varepsilon\in(0,1)$, such that   $\forall x,y\in\X$,  
    $\mathcal W(t(\cdot|x),t(\cdot|y))\leqslant (1-\varepsilon) \ell(x,y)$.
\end{assumptionT}
In the case of the DULA and the DUPS, the contraction is obtained when the target is regular or when the step-size is small, see \eqref{eq:contraction-dula}, \Cref{t:contraction-dula-etasmall}, \eqref{t:contraction-dups} and \Cref{t:contraction-DUPS-eta-small}.
Besides, our convergence result requires the following assumptions on the acceptance rate $A_z(\cdot|\cdot)$. They ensure in particular that enough proposal samples are accepted.
\begin{assumptionA}\label{A:Lipschitz}
    For $x,y\in\X$, if $z_1\sim u(\cdot|x)$, $z_2\sim u(\cdot|y)$ and $x'\sim v(\cdot|z_1)$, $y'\sim v(\cdot |z_2)$, we have
    $\bE|A_{z_1}(x'|x)-A_{z_2}(y'|y)|\leqslant L\,\ell(x,y)$. 
\end{assumptionA}
\begin{assumptionA}\label{A:lower-bound}
    For $x\in\X$, if $z\sim u(\cdot|x)$ and $x'\sim v(\cdot|z)$, then $\bE[A_z(x'|x)]\geqslant 1-\delta$.
\end{assumptionA}
Assumption \Cref{A:Lipschitz} is a regularity assumption of the acceptance rate, while assumption \Cref{A:lower-bound} ensures that sufficiently many samples are accepted through the sampling process to have faster convergence, and can be seen as an implicit measure of proximity between the invariant probability distribution of $t(\cdot|\cdot)$ and the target distribution $p(\cdot)$.

We now state our result, which gives an upper bound on the contraction rate for a Metropolis-adjusted sampler.
We illustrate in \Cref{app:DMAPS} how this applies to DMAPS.

\begin{theorem}\label{p:contraction-generic}
    Suppose \Cref{T:contraction} and \Cref{A:Lipschitz}-\Cref{A:lower-bound}.
    Denote $D=\mathrm{diam} (\X)$.
    Then,
    for $\varepsilon\in(0,1)$ satisfying \Cref{T:contraction}, we have
    \begin{equation}
        \label{eq:contraction-W}
       \forall x,y \in \X, \qquad \mathcal W\left(t_{\rm MH}(\cdot|x) , t_{\rm MH}(\cdot|y)\right) \leqslant \left( 1- \varepsilon + \delta + LD \right) \ell(x,y).
    \end{equation}
\end{theorem}
\begin{proof}
    Consider 
    $z_1 \sim u(\cdot|x)$, $z_2\sim u(\cdot |y)$, and $x'\sim v(\cdot|z_1)$, $y'\sim v(\cdot|z_2)$, so that $\mathcal W(t(\cdot|y),t(\cdot|z))=\bE[\ell(y',z')]$. Consider $U$ an uniform random variable on $[0,1]$ and define
    $$
    x'' = \left\{ \begin{array}{ll}
         x' & \text{if } U\leqslant A_{z_1}(x'|x) \\
         x & \text{otherwise,}
    \end{array} \right.
 \   \mbox{  
    and likewise } \
    y'' = \left\{ \begin{array}{ll}
         y' & \text{if } U\leqslant A_{z_2}(y'|y) \\
         y & \text{otherwise,}
    \end{array} \right.
    $$
    so that $x''\sim t_{\rm MH}(\cdot|x)$ and $y''\sim t_{\rm MH}(\cdot|y)$.
    We have then
    \begin{multline*}
        \mathcal W(t_{\rm MH}(\cdot|x),t_{\rm MH}(\cdot|y) ) 
        \leqslant \bE[\ell(x'',y'')] \\
        = \bE\left[ \1_{U\leqslant \min\{A_{z_1}(x'|x),A_{z_2}(y'|y)\}} \ell(x',y') \right] 
         + \bE\left[ \1_{U> \max\{A_{z_1}(x'|x),A_{z_2}(y'|y)\}} \ell(x,y) \right] \\
         + \bE\left[ \1_{A_{z_1}(x'|x)<U\leqslant A_{z_2}(y'|y)} \ell(x,y') \right] 
         + \bE\left[ \1_{A_{z_2}(y'|y)<U\leqslant A_{z_1}(x'|x)} \ell(x',y) \right]  \\
         = \bE\left[ \min\{A_{z_1}(x'|x),A_{z_2}(y'|y)\}\ell(x',y') \right] 
         + \bE\left[  \left(1-\max\{A_{z_1}(x'|x),A_{z_2}(y'|y)\}\right) \ell(x,y) \right] \\
         + \bE\left[ [A_{z_1}(x'|x)-A_{z_2}(y'|y)]^+ \ell(x,y') \right] 
         + \bE\left[ [A_{z_1}(x'|x)-A_{z_2}(y'|y)]^- \ell(x',y) \right] \\
         \leqslant \bE\left[ \min\{A_{z_1}(x'|x),A_{z_2}(y'|y)\}\ell(x',y') \right] 
         + \delta\, \ell(x,y)
         + D \, \bE\left| A_{z_1}(x'|x)-A_{z_2}(y'|y) \right|,
    \end{multline*}
    where we used \Cref{A:lower-bound} to obtain the last inequality.
    Plus, by \Cref{A:Lipschitz}, $\bE\left| A_{z_1}(x'|x)-A_{z_2}(y'|y) \right|\leqslant L\,\ell(x,y)$.
      Hence, assumption \Cref{T:contraction} gives us:
     $
      \mathcal W(t_{\rm MH}(\cdot|x),t_{\rm MH}(\cdot|y ) )
      \leqslant  \left(1-\varepsilon + \delta + LD \right) \ell(x,y).
      $
\end{proof}

\section{Experiments} \label{sec:experiments-dula-DUPS}
            One possibility to evaluate the performances of sampling algorithms is to analyze their $2^d \times 2^d$ transition matrices when $d$ is small enough. In particular, we are interested in two characteristics.
            The first one is the relaxation time (which behaves like the mixing time) $T_{\rm rel}(t(\cdot|\cdot))$ of a transition kernel $t(\cdot|\cdot)$, defined by
            \begin{equation}
                T_{\rm rel}(t(\cdot|\cdot)) = \frac{1}{1-\lambda_2(t(\cdot|\cdot))} .
            \end{equation}
            The second one is the distance of the invariant probability distribution of the kernel $t(\cdot|\cdot)$ to the target distribution $p(\cdot)$.

            For each target distribution $p(\cdot)$ presented below, we test DULA and DUPS, as well as their adjusted variants DMALA and DMAPS, and compare their performances between them and with Gibbs sampling. Besides, we consider three different scores:
            \begin{itemize}
                \item[(i)] the Stein score $s_{\rm Stein}(\cdot)$ obtained by taking the gradient of a natural continuation of $\log p(\cdot)$ on $\mathbb R^d$;
                \item[(ii)] the Gibbs score $s_{\rm Gibbs}(\cdot)$ defined in \eqref{eq:gibbs-score} for which  DULA approximates an exact dynamic when the step-size tends to zero;
                \item[(iii)] the Glauber score $s_{\rm Glauber}(\cdot) = \delta\log p(\cdot)$  for which DUPS approximates an exact dynamics when the step-size tends to zero.
            \end{itemize}

            \paragraph{(Mixture of) independent bits.}
            Consider the following probability distribution on $\{-1,1\}^d$ which models the distribution of $d$ independent bits, where each bit follows a Rademacher distribution with parameter $\sigma(2\beta)$ for some $\beta\in\mathbb R$:
                \begin{equation}
                    p_{\rm bits} (x) = \frac{1}{(2\cosh(\beta))^d} \exp\left( \beta \cdot \mathbf 1^\top x \right).
                \end{equation}
                As the potential of $p_{\rm bits}(\cdot)$, the standard score et the natural score are both constant and equal to $\beta\mathbf 1$.
                Besides, DUPS with the standard score is ``perfect,'' in the sense that its invariant probability distribution is equal to the target distribution.
                We also consider the following mixture model, for which we expect a higher difficulty for all samplers.
                \begin{equation}
                    p_{\rm mixture} (x) = \frac{1}{2}p_{\rm bits} (x) + \frac{1}{2}p_{\rm bits} (-x).
                \end{equation} 
                We display our results for the mixture model in \Cref{fig:TVindbits}. The DULA approximates well the target for small step-sizes and the Gibbs score improves the performances.
                Remarkably, the observations confirm the results in \Cref{sec:unadjusted}: DULA with Gibbs score provides the best approximation for small step-sizes of the target distribution among the unadjusted algorithms, while DUPS converges faster. These observations hold for every model tested in our experiments.
                For the mixture model, DUPS with standard score seems the best approximate sampler, with a good approximation of the target even for large step-sizes and a lower mixing time (by a factor 100) than Gibbs sampling.
                For the adjusted algorithms, it seems that  DMAPS outperforms  DMALA and Gibbs sampling, for all choices of scores, with a slightly better convergence speed with the Stein score.

                        \begin{figure}[h]
                            \centering
                            
                            \hspace*{-.25cm}\includegraphics[width=0.33\linewidth]{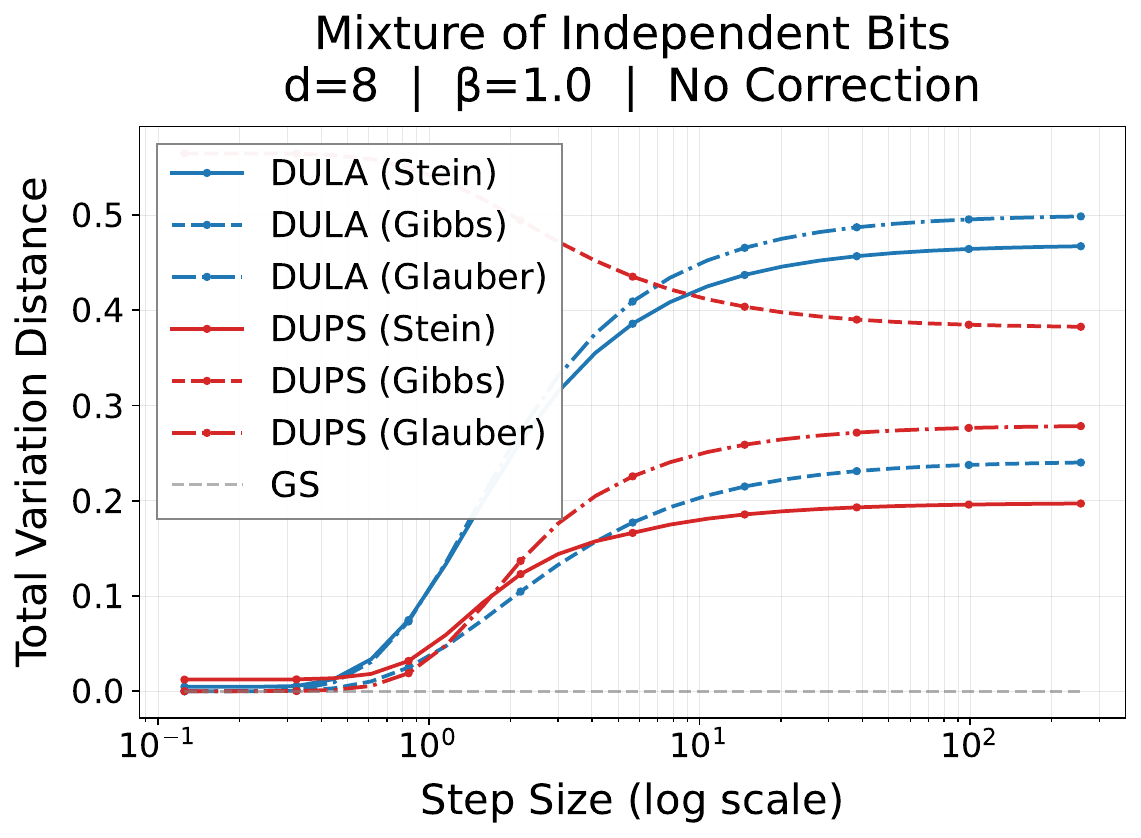}
                            \includegraphics[width=0.33\linewidth]{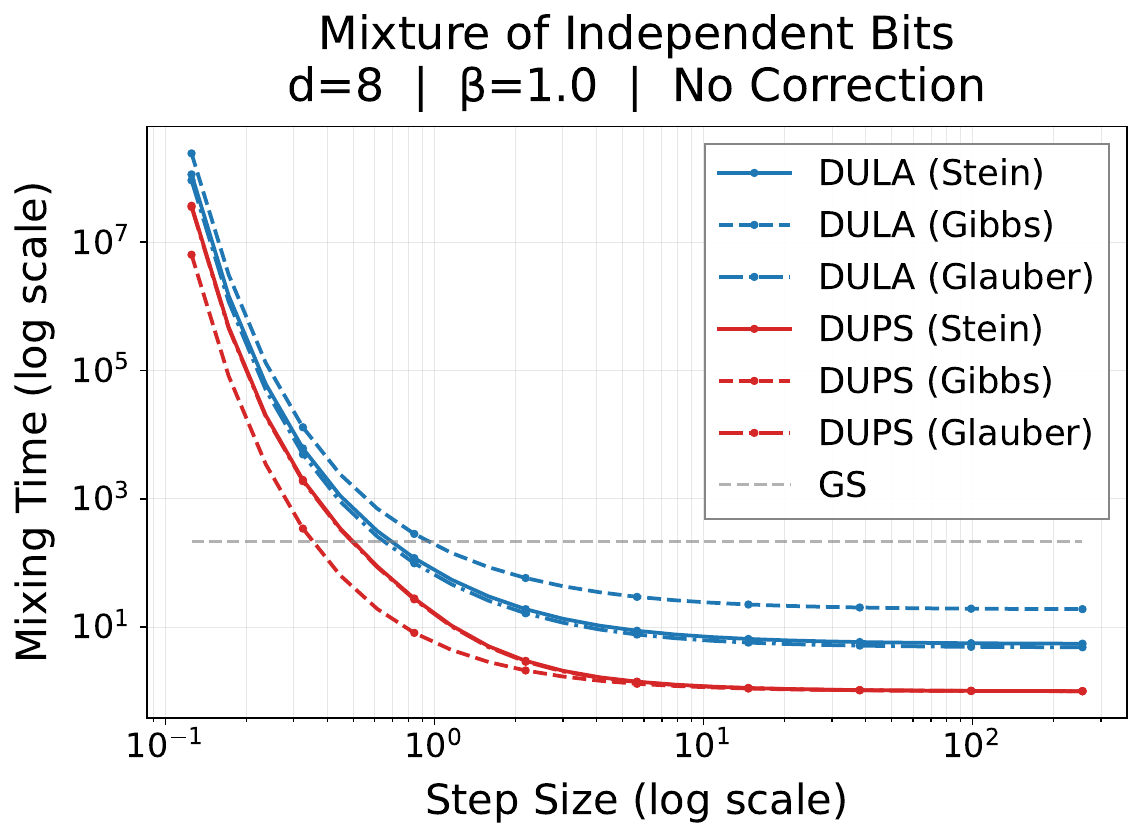} \includegraphics[width=0.33\linewidth]{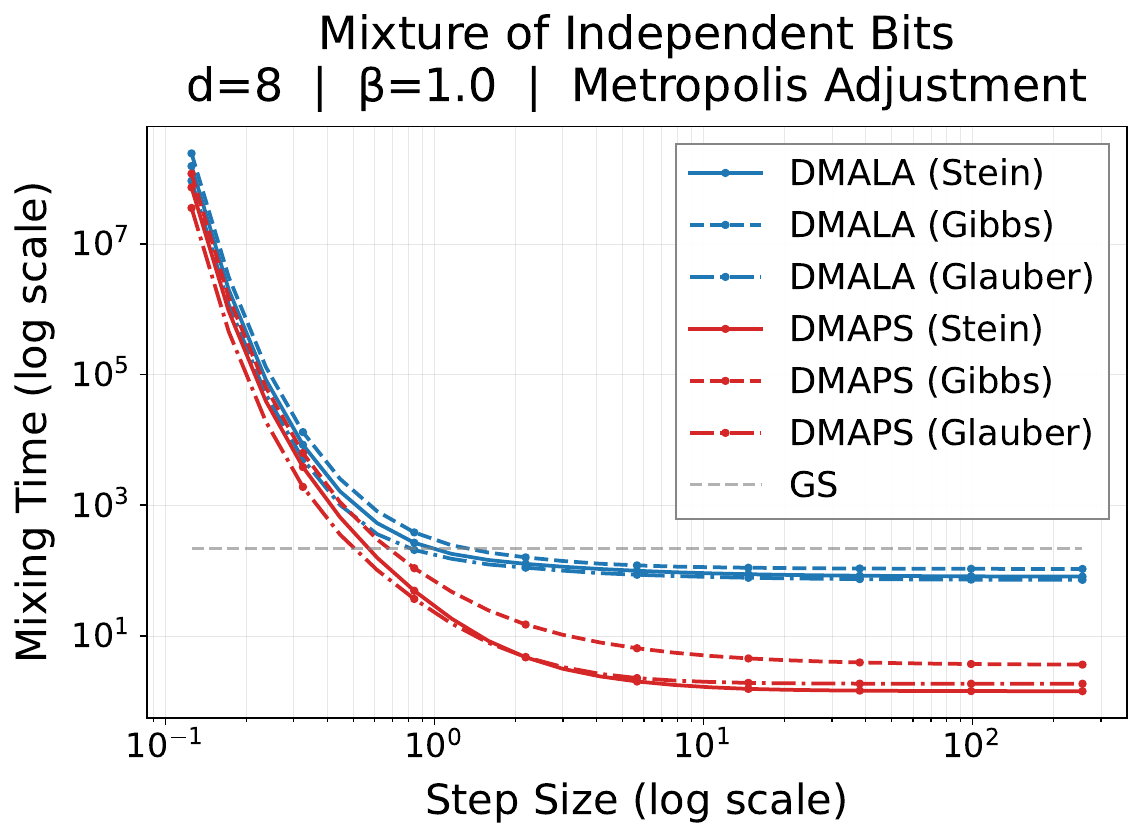} \hspace*{-.25cm}

                            \caption{Left: Distance to the target for DULA, DUPS and Gibbs sampling in the mixture model. Center: Mixing times for DULA, DUPS and Gibbs sampling in the mixture model. Right: Mixing times for DMALA, DMAPS and Gibbs sampling in the mixture model.}
                            \label{fig:TVindbits}
                        \end{figure}

            \paragraph{Ising model.}
            We define the distribution of the Ising model~\citep{ising1925beitrag, friedli2017statistical} with interaction parameter $J\in\mathbb R$ and external magnetic field $h\in\mathbb R$ as
                \begin{equation}
                    p_{\rm Ising} (x) \propto \exp\left( J \sum_{d(i, j)=1} x_ix_j + h \cdot \mathbf 1^\top x \right) .
                \end{equation}
                We observe that the Stein score and the Glauber score are equal. For our experiments, we consider only ferromagnetic interactions $J>0$, which correspond to the most difficult target distributions for our samplers. We test the Ising model on a grid of size $3\times 3$. \Cref{fig:ising} shows that DULA offers better approximations of the target distribution while the DUPS converges faster. For instance, with a step-size $\eta\approx 0.4$, the invariant distribution of the DUPS is very close to the target with a mixing time about $10^3$ smaller than those of Gibbs sampling and DULA.
                Finally, DMAPS is once again the fastest algorithm, especially with Stein/Glauber scores.

                \begin{figure}[h]
                            \centering
                            
                            \hspace*{-.25cm}\includegraphics[width=0.33\linewidth]{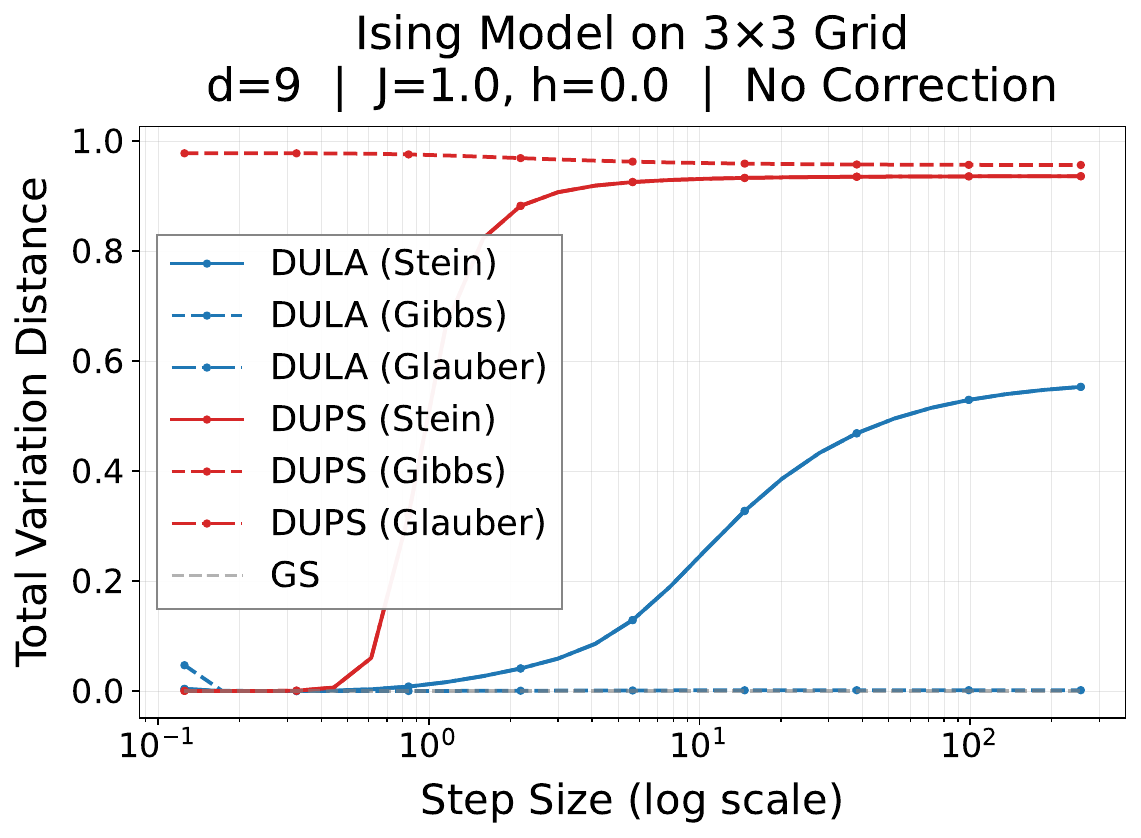} 
                            \includegraphics[width=0.33\linewidth]{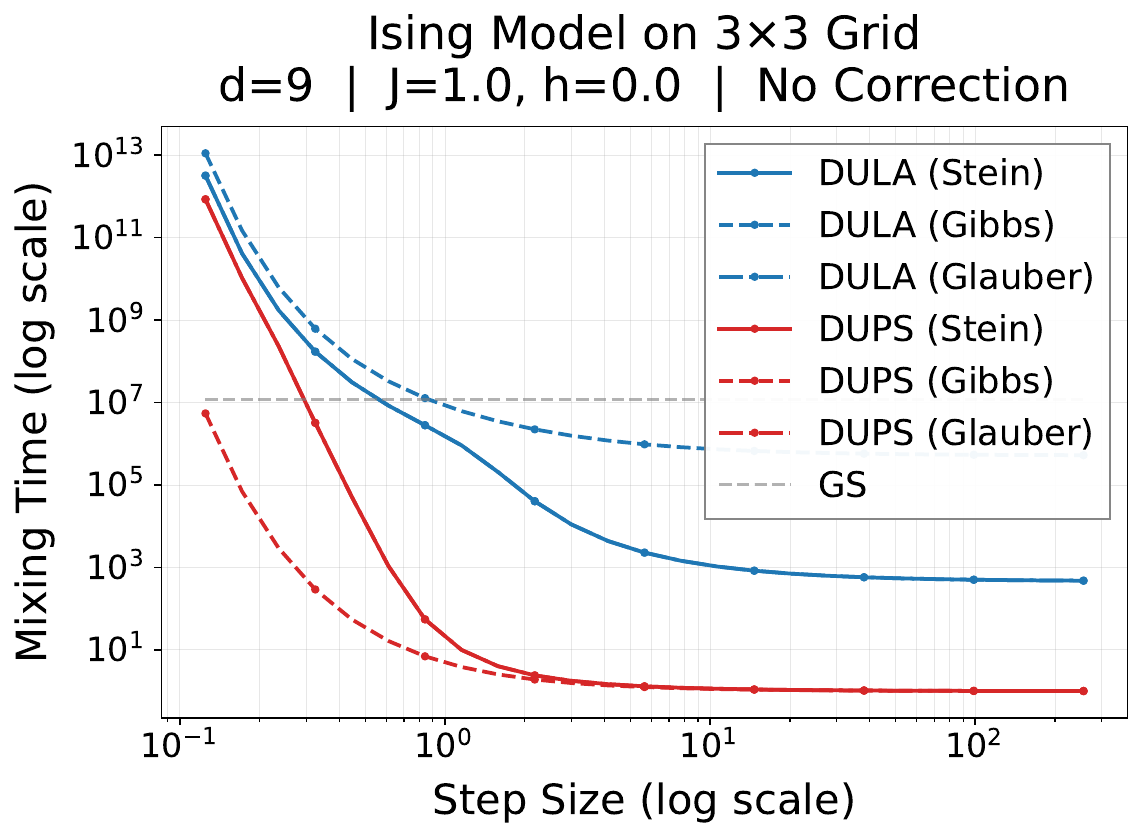}
                            \includegraphics[width=0.33\linewidth]{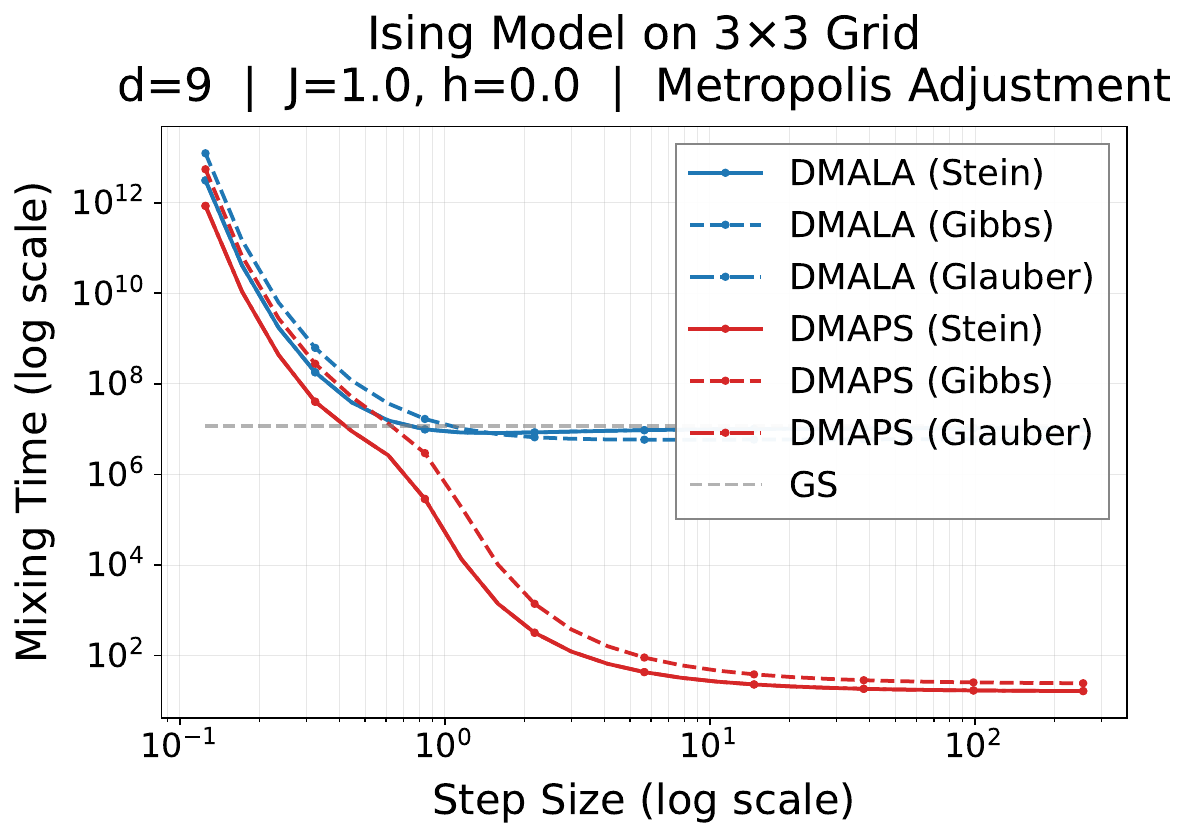} \hspace*{-.25cm}

                          \caption{Distance to the target and mixing times for DULA, DUPS and Gibbs sampling in the Ising model.}
                            \label{fig:ising}
                        \end{figure}

            \paragraph{Curie-Weiss model.} 
            An alternative to the Ising model is the Curie-Weiss model~\citep{curie1895proprietes,weiss1907hypothese, friedli2017statistical}, given by the following distribution:
            \begin{equation}
                p_{\rm Curie-Weiss}(x) \propto \exp\left( \beta  \left( \mathbf 1^\top x -b \right)^2 \right).
            \end{equation}
            However, it becomes a very hard sampling problem with a high ferromagnetic interaction, i.e., when $\beta>0$ becomes large. 
            In our experiments in \Cref{fig:curieweiss} with $\beta=1$, the transition matrices are most likely highly ill-conditioned so that the invariant measure is not right for Gibbs sampling. Yet, we can see that only DUPS and DMAPS have good mixing times.
            We also test a small positive value for $\beta$ as well as a negative value (antiferromagnetic interaction). In \Cref{fig:curieweiss}, we observe similar behaviors than on the experiments on the Ising model for ferromagnetic interactions: the DULA gives better approximations of the target distribution (especially with Gibbs score) while the DUPS is faster. 
            For the antiferromagnetic interactions, DULA and DUPS are similar.
             DMAPS has better performances than DMALA and Gibbs sampling.
            \begin{figure}[h]
                            \centering
                            
                            \hspace*{-.25cm}\includegraphics[width=0.33\linewidth]{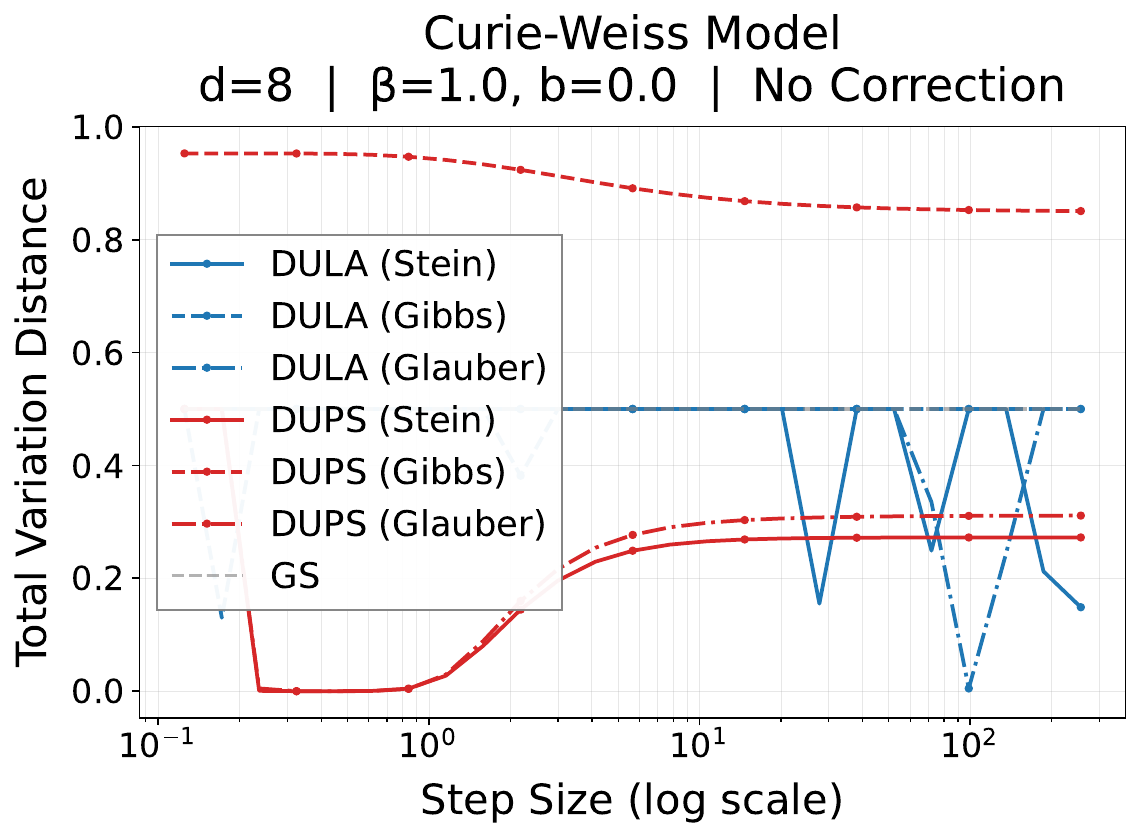} 
                            \includegraphics[width=0.33\linewidth]{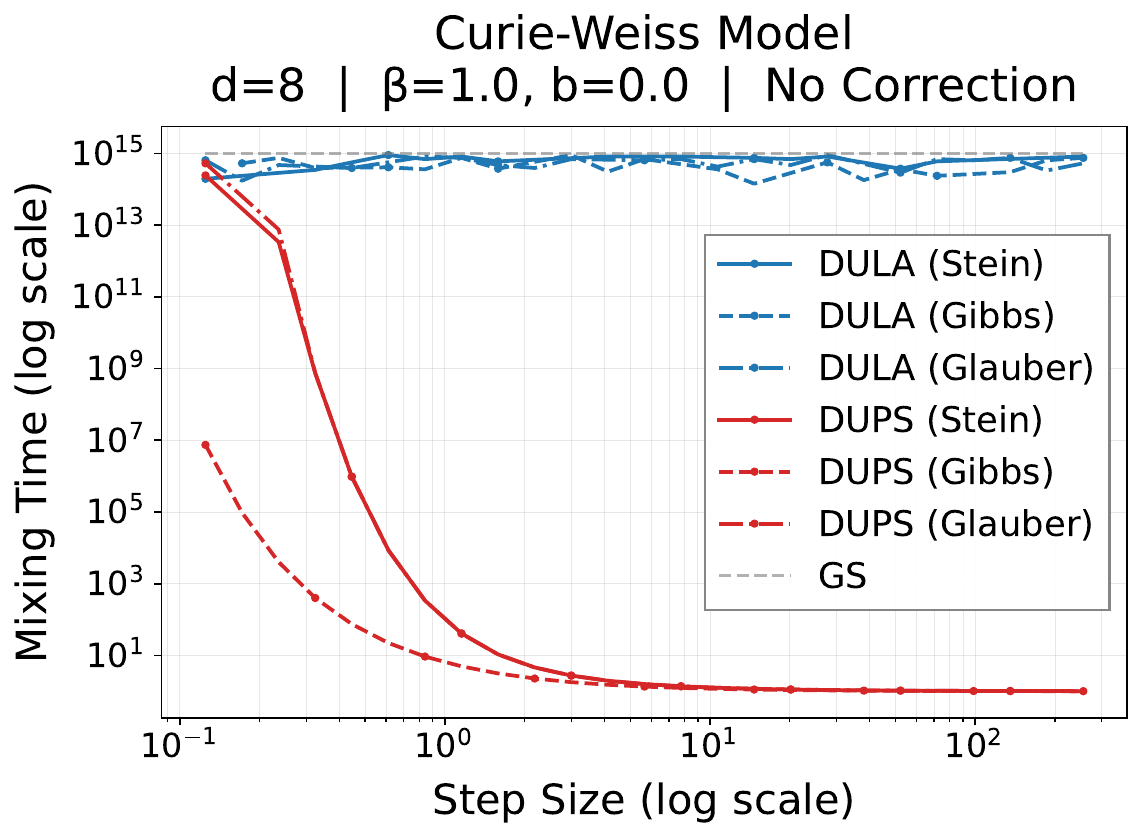}
                            \includegraphics[width=0.33\linewidth]{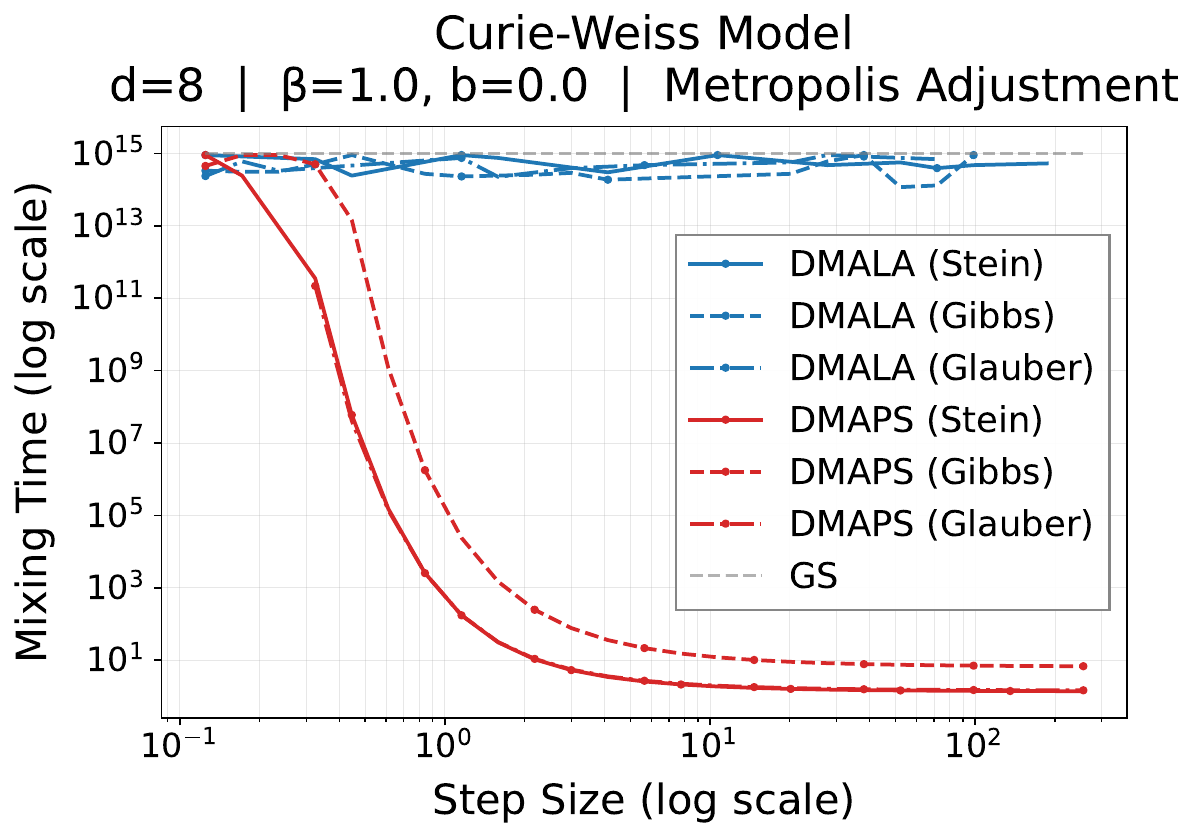}\hspace*{-.25cm}
                            
                            \hspace*{-.25cm}\includegraphics[width=0.33\linewidth]{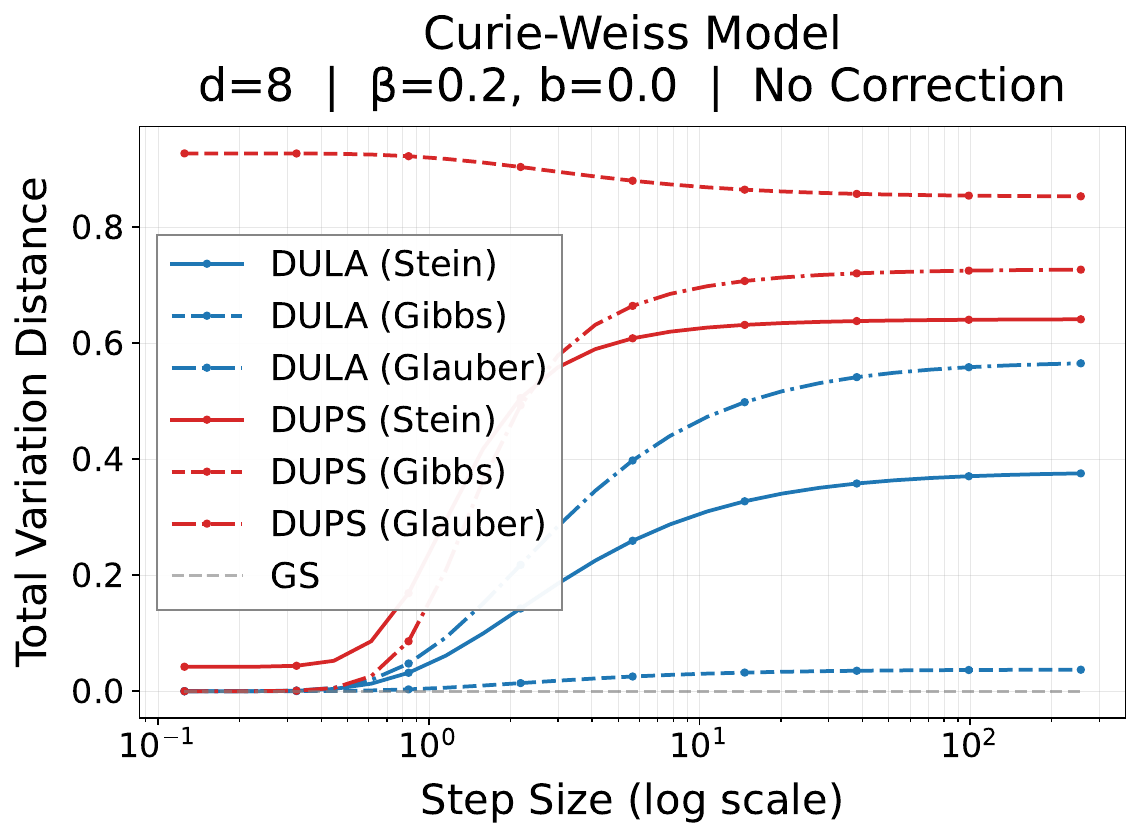} 
                            \includegraphics[width=0.33\linewidth]{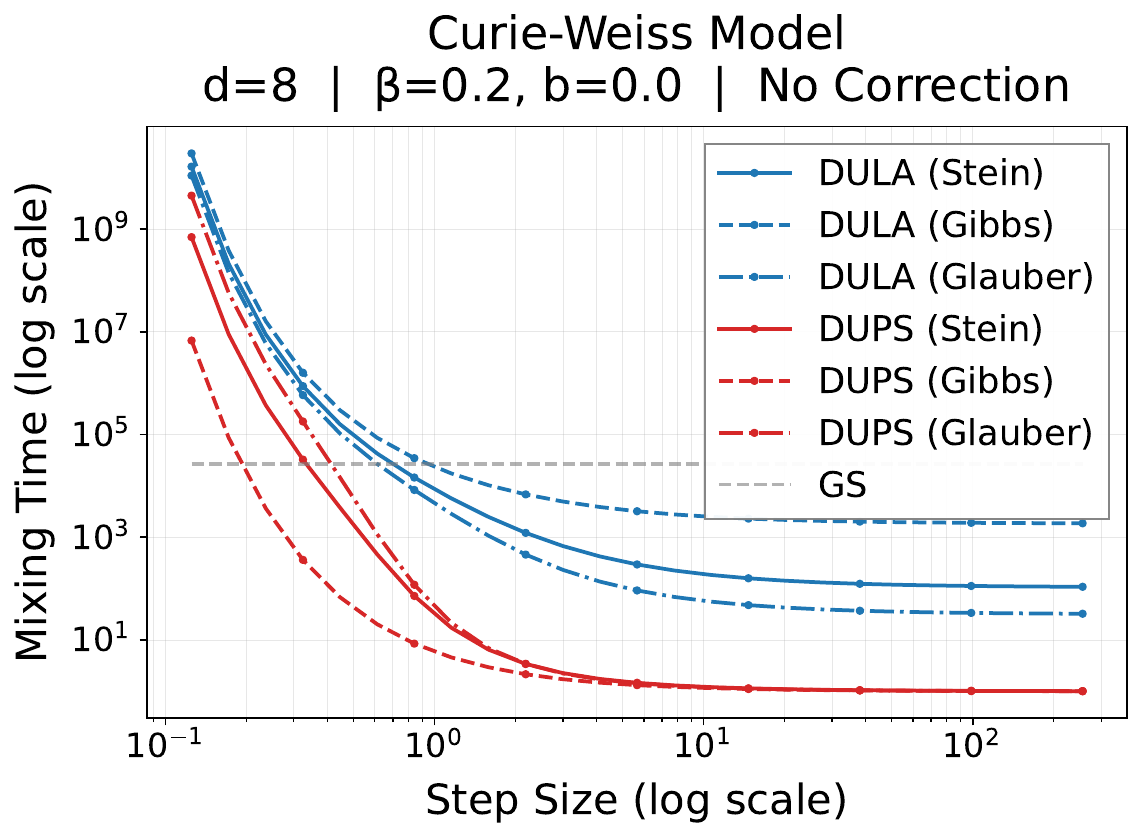}
                            \includegraphics[width=0.33\linewidth]{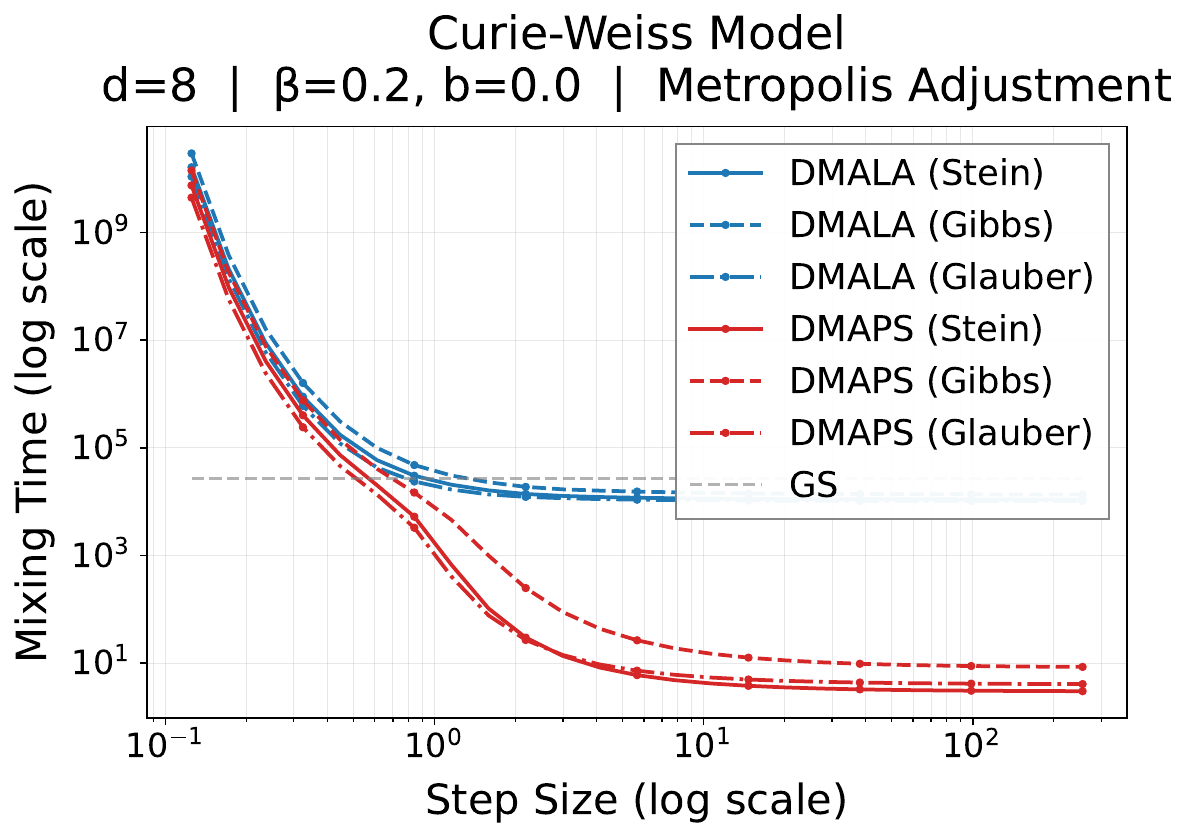} \hspace*{-.25cm}
                            
                            \hspace*{-.25cm}\includegraphics[width=0.33\linewidth]{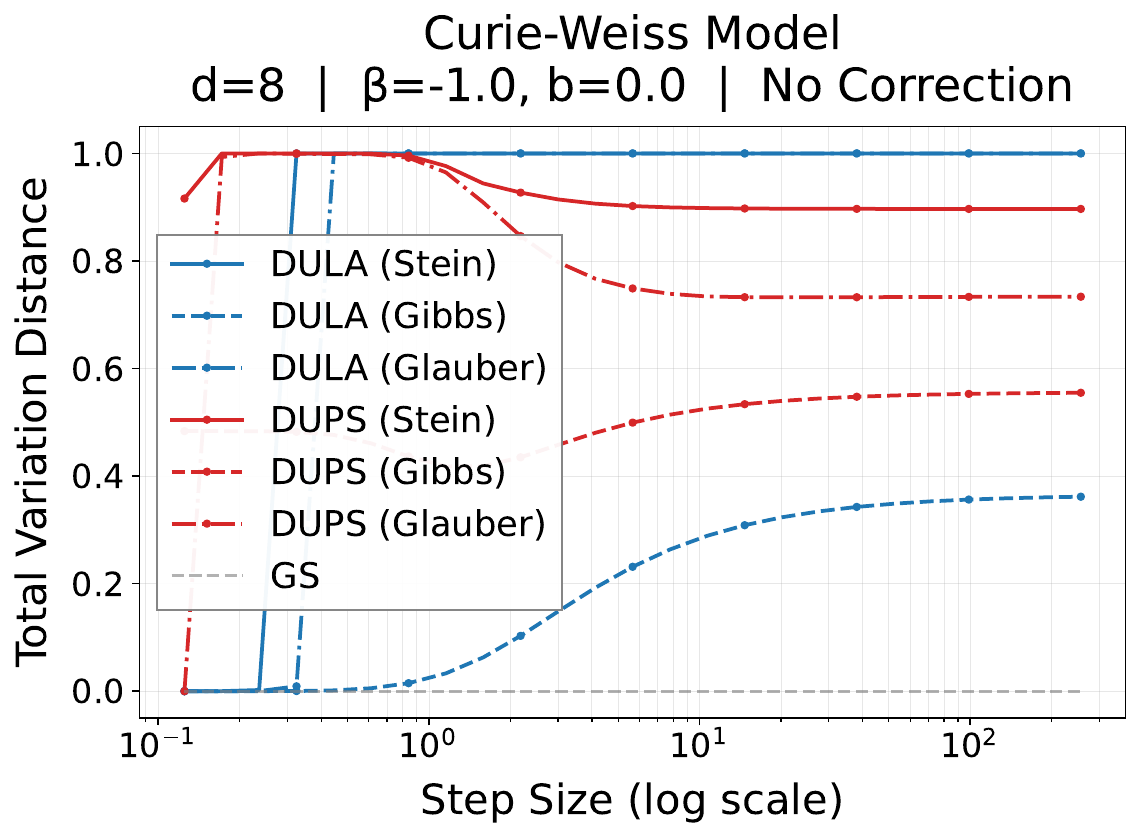} 
                            \includegraphics[width=0.33\linewidth]{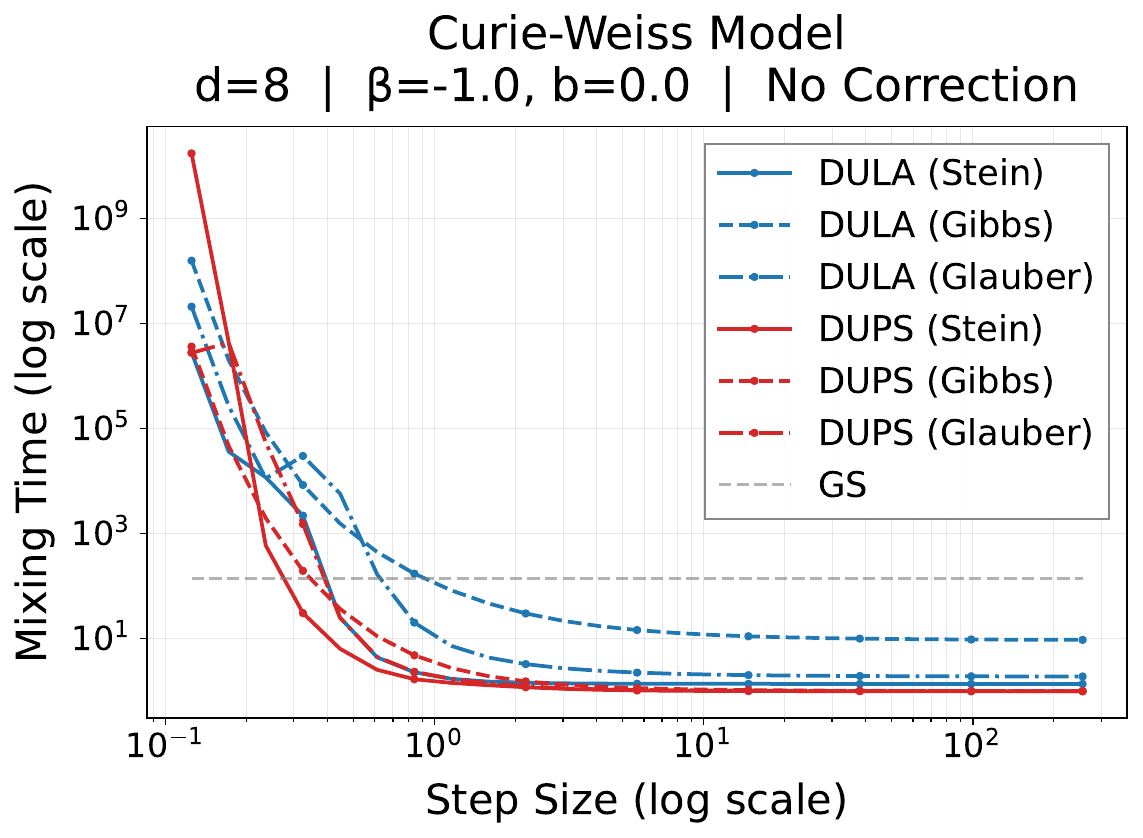}
                            \includegraphics[width=0.33\linewidth]{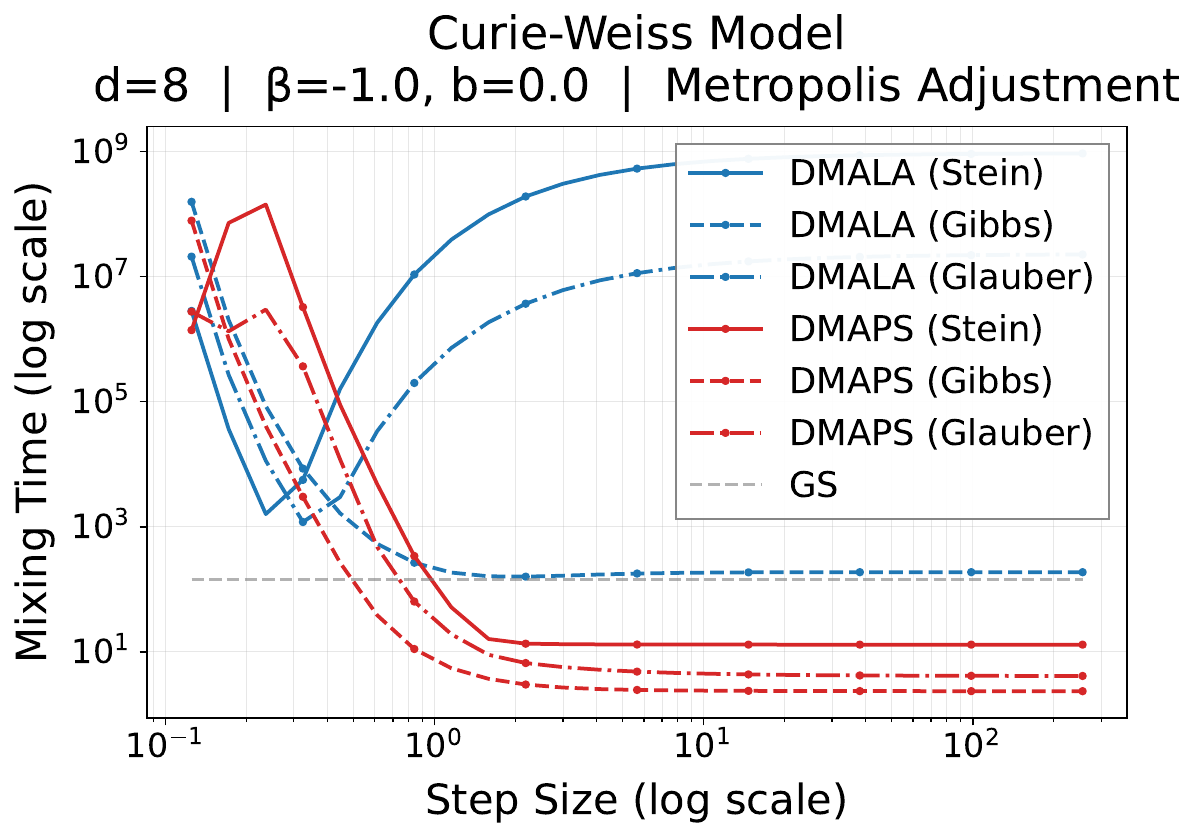} \hspace*{-.25cm}

                            \caption{Distance to the target and mixing times for DULA, DUPS and Gibbs sampling in the Curie-Weiss model.}
                            \label{fig:curieweiss}
                        \end{figure}

\section{Conclusion}
In this paper, we showed our discrete Langevin algorithms, once using adjusted scores, could be seen as the discretization of a continuous-time jump processes, with improved convergence behaviors. We provided theoretical results on their contraction properties and approximation of the target distribution, and provided experiments showing that they outperform traditional Gibbs sampling. While we focused on the binary hypercube, extensions to more general factorized state-spaces are worth pursuing, together with applications to generative diffusion-based or denoising-based models.

\subsection*{Acknowledgments}
This work has received support from the French PEPR integrated project HQI (ANR-22-PNCQ-0002), and the French National Research Agency, under the France 2030 program with the reference “PR[AI]RIE-PSAI”
(ANR-23-IACL-0008).

\bibliographystyle{plainnat}
\bibliography{bibliography}

\clearpage

\appendix

\crefalias{section}{appendix}

\section{Proof of \Cref{t:gibbs}}\label{app:gibbs}

    Let $x,y\in\{-1,1\}^d$ such that $x_1\neq y_1$, and consider two samples $x',y'$  from $t_{\rm GS}(\cdot|x),t_{\rm GS}(\cdot|y)$, respectively, such that both their flipped coordinate is $1$ (with probability $d e^{-2/\eta}$). Then,
    \begin{align*}
        \mathcal W (  t_{\rm GS}(\cdot|x),t_{\rm GS}(\cdot|y) ) &
        \leqslant
        \ell (x,y) - d e^{-2/\eta} \times (1 - \bP [x'_1\neq y'_1]) \\
        & = 
        \ell (x,y) -d e^{-2/\eta} \times\left( 1- \left|  \sigma( 2 \delta\log p(x)_1 ) -  \sigma( 2 \delta\log p(y)_1 ) \right| \right) \\
        & \leqslant
        \ell (x,y) -d e^{-2/\eta} \times\left( 1- \beta_2 \times \ell(x,y) \right) \\
        & \leqslant
        \left( 1 + d e^{-2/\eta}\beta_2  - e^{-2/\eta} \right)\times \ell(x,y) = \left( 
            1 - e^{-2/\eta}(1-d\beta_2)
        \right) \times \ell(x,y) .
    \end{align*}

\section{Proof for \Cref{sec:DULA}}\label{app:proofs-dula}

\subsection{Proof of \Cref{t:discretization-dula}}\label{app:discretization-dula}

    We will prove that, for $x,x'\in\X$, when $\eta\to0$ we have that $t_{\rm DULA}(x'|x)$ behaves like
    \begin{equation}\label{eq:TA-DULA}
        \left\{
        \begin{array}{ll}
            1-e^{-2\eta} \sum_{i=1}^d \exp(-x_i s(x)_i)   &  \text{if } x'=x \\
            e^{-2/\eta} \exp (-x_i s(x)_i) & \text{if } x' = [-x_i,x_{-i}]  \\
            0 & \text{otherwise.}
        \end{array}
        \right.
    \end{equation}
    Indeed, given a fixed $x\in\X$ we have that, for every $x'\in\X$:
    \begin{align*}
        t_{\rm DULA} (x'|x) & \propto \exp\left( \frac{x^\top x'}{\eta} + \frac{s(x)^\top x'}{2} \right) \\
        & = \exp\left(\frac d\eta + \frac{s(x)^\top x}{2} \right) \times \prod_{i\mid x_i'\neq x_i} e^{-2/\eta} \exp(-x_i s(x)_i) \\
        & \propto e^{-2\ell(x,x')/\eta} \prod_{i\mid x_i'\neq x_i}  \exp(-x_i s(x)_i) .
    \end{align*}
    Therefore, using that $(1+x)^{-1}=1-x+O_{x\to 0^+}(x^2)$, we obtain that
    \begin{equation*}
        \begin{array}{ll}
            t_{\rm DULA} (x'|x) = 1-e^{-2/\eta} \sum_{i=1}^d \exp(-x_i s(x)_i) + \underset{\eta\to0}{O}\left( e^{-4/\eta} \right)&   \text{if } x'=x \\
            t_{\rm DULA} (x'|x) = e^{-2/\eta} \exp (-x_i s(x)_i) + \underset{\eta\to0}{O}\left( e^{-4/\eta} \right)& \text{if } x' = [-x_i,x_{-i}]  \\
            t_{\rm DULA} (x'|x) = \underset{\eta\to0}{O}\left( e^{-4/\eta} \right) & \text{otherwise.}
        \end{array}
    \end{equation*}
    Besides, the transition kernel \eqref{eq:TA-DULA} is the discretization of the Markov process with generator matrix \eqref{eq:generator-DULA}.
    Furthermore, it equals the transition kernel of Gibbs sampling \eqref{eq:transition-gibbs} when the score $s(\cdot)$ is given by \eqref{eq:gibbs-score}
    and $e^{-2/\eta}\leqslant1/d$.

\subsection{Proof of \Cref{t:contraction-dula-etasmall}}\label{app:contraction-dula-etasmall}

    Following the proof of Proposition 3.1 by \cite{bach2025sampling}, given any $x,y\in\X$, we can
 compute an upperbound of (using that $x_i s(x)_i$ and $y_i s(y)_i$ are non negative, but not necessarily exactly obtained from the Glauber score $\delta \log p$)
\begin{align*}
\mathcal W & ( t_{\rm DULA}(\cdot|x), t_{\rm DULA}(\cdot,y) )
  \leqslant \sum_{i=1}^d \left| \sigma\left( \frac{2}{\eta} x_i +s(x)_i\right) - \sigma\left( \frac{2}{\eta} y_i +s(y)_i\right) \right| \\
&  = \sum_{x_i = y_i} \left| \sigma\left( \frac{2}{\eta} + x_i s(x)_i\right) - \sigma\left( \frac{2}{\eta} + y_i s(y)_i\right) \right|  \\
& \qquad+ \sum_{x_i = -y_i}\left| \sigma\left( \frac{2}{\eta} + x_i s(x)_i\right) - \sigma\left( \frac{2}{\eta} x_i y_i + x_i s(y)_i\right) \right|
\\
&  \leqslant \frac{1}{2 + e^{2/\eta}} \sum_{x_i = y_i} \left| s(x)_i - s(y)_i \right| 
  + \sum_{x_i = -y_i}\left| 1 - \sigma\left( -\frac{2}{\eta} - x_i s(x)_i\right) - \sigma\left( -\frac{2}{\eta} - y_i s(y)_i\right) \right|
\\
 & \leqslant \frac{d \beta_2 \left\| x - y \right\|_1 }{2 + e^{2/\eta}} + \sum_{x_i = -y_i} \left( 1 - \frac{1}{2} e^{ -\frac{2}{\eta} - x_i s(x)_i} - \frac{1}{2} e^{ -\frac{2}{\eta} - y_i s(y)_i} \right)
\\
 & = \frac{d \beta_2 \left\| x - y \right\|_1 }{2 + e^{2/\eta}} + \sum_{x_i = -y_i} \left\{ 1 - \frac{1}{2} e^{ -\frac{2}{\eta}} \left( e^{ - x_i s(x)_i} + e^{ - y_i s(y)_i} \right) \right\}
\\
&  = \frac{d \beta_2 \left\| x - y \right\|_1 }{2 + e^{2/\eta}} + \sum_{x_i = -y_i} \left\{ 1 - \frac{1}{2} e^{ -\frac{2}{\eta}} \left( \sigma( -2 x_i \delta \log p(x)_i) + \sigma( -2 y_i \delta \log p(y)_i) \right) \right\} .
\end{align*}
We decompose the sum of the sigmoid terms as follows
\begin{multline*}
    \sigma( -2 x_i \delta \log p(x)_i) + \sigma( -2 y_i \delta \log p(y)_i) =
    \sigma( -2 x_i \delta \log p(x)_i) - \sigma( -2 x_i \delta \log p(y)_i) \\
    + \sigma( -2 x_i \delta \log p(y)_i) + \sigma( -2 y_i \delta \log p(y)_i)  .
\end{multline*}
We obtain therefore that
\begin{align*}
\mathcal W & ( t_{\rm DULA}(\cdot|x), t_{\rm DULA}(\cdot,y) ) \\
& \leqslant \frac{d \beta_2 \left\| x - y \right\|_1 }{2 + e^{2/\eta}} + \sum_{x_i = -y_i} \left\{ 1 - \frac{1}{2} e^{ -\frac{2}{\eta}} \left( \sigma( -2 x_i \delta \log p(x)_i) - \sigma( -2 x_i \delta \log p(y)_i) + 1\right) \right\}
\\
 & \leqslant \frac{d \beta_2 \left\| x - y \right\|_1 }{2 + e^{2/\eta}} + \sum_{x_i = -y_i} \left\{ 1 - \frac{1}{2} e^{ -\frac{2}{\eta}} \left(-\frac{1}{2} \beta_2 \left\| x - y \right\|_1 + 1\right) \right\}
\\
 & \leqslant \frac{d \beta_2 \left\| x - y \right\|_1 }{2 + e^{2/\eta}} + \ell (x,y) \times \left( 1 - \frac{1}{2} e^{ -\frac{2}{\eta}} \left(-\frac{1}{2} \beta_2 \left\| x - y \right\|_1 + 1\right)\right)\\
&  \leqslant \frac{d \beta_2 \left\| x - y \right\|_1 }{2 + e^{2/\eta}} + \ell (x,y) \times \left( 1 - \frac{1}{2} e^{ -\frac{2}{\eta}} \left( - \beta_2 d + 1\right)\right)
\\
 & =\left( 1 - e^{-2/\eta} \left( \frac{1}{2} - \frac{3}{2} \beta_2 d \right)\right) \times \ell (x,y).
\end{align*}
Thus, if $\beta_2 d \leqslant \frac{1}{4}$, we get a contraction factor of $(1 - \frac{1}{4} e^{-2/\eta})$.

\subsection{Proof of \Cref{t:error-DMALA-eta-small}}\label{app:error-DMALA-eta-small}

    We have :
    \begin{align*}
        \mathcal W(p,\hat p_{\rm DULA}) & = \mathcal W(p t_{\rm GS}(\cdot), \hat p_{\rm DULA} t_{\rm DULA}(\cdot)) \\
        & \leqslant \mathcal W(p t_{\rm DULA}(\cdot), \hat p t_{\rm DULA}(\cdot)) + \mathcal W(p t_{\rm GS}(\cdot), p t_{\rm DULA}(\cdot)) \\
        & \leqslant \left[ 1-\frac14 \exp(-2/\eta) \right] \times \mathcal W(p,\hat p_{\rm DULA}) \\
       &\qquad \qquad + \bE_{x\sim p}\left[ \inf_{x'\sim t_{\rm DULA}(\cdot|x), y'\sim t_{\rm GS}(\cdot|x)}  \bE\left[ \ell (x',y') \mid x\right]  \right] .
    \end{align*}
    Let $x\in\{-1,1\}^d$. We have, using that $x_is(x)_i\geqslant0$:
    \begin{align*}
        \inf_{x'\sim t_{\rm DULA}(\cdot|x), y'\sim t_{\rm GS}(\cdot|x)}  \bE\left[ \ell (x',y') \mid x\right]
        & = \inf_{x'\sim t_{\rm DULA}(\cdot|x), y'\sim t_{\rm GS}(\cdot|x)} \sum_{i=1}^d \bP [x'_i\neq y'_i] \\
        & = \sum_{i=1}^d \left| \bP_{x'\sim t_{\rm DULA}(\cdot|x)}[x'_i=- x_i] -  \bP_{y'\sim t_{\rm GS}(\cdot|x)}[y'_i = - x_i]  \right| \\
        & = \sum_{i=1}^d \left| \sigma\left( -\frac2\eta -x_is(x)_i \right)   - \exp \left( -\frac2\eta -x_is(x)_i \right)   \right| \\
        & = \sum_{i=1}^d \frac{\exp\left( -\frac2\eta -x_is(x)_i \right)}{1 + \exp \left( \frac2\eta +x_is(x)_i \right)}   \leqslant d \frac{\exp\left( -\frac2\eta \right)}{1+\exp\left( \frac2\eta  \right)} .
    \end{align*}
    Therefore,
    \begin{align*}
        \mathcal W(p,\hat p_{\rm DULA}) & \leqslant 4 \exp(2/\eta) \times d \frac{\exp\left( -\frac2\eta \right)}{1+\exp\left( \frac2\eta \right)} =  \frac{4d}{1+ e^ {2/\eta}}.
    \end{align*}

\section{Proofs for \Cref{sec:DUPS}}\label{app:proofs-dups}

    \subsection{Proof of \Cref{t:contraction-dups}}\label{app:contraction-dups}

We consider the Wasserstein contraction effects of the two stages. 

For the transition kernel $u(\cdot|\cdot)$, considering two starting states $x_1,x_2\in\X$, we get a Wasserstein distance less than (using the same arguments as \cite{bach2025sampling}, detailed again for completeness) 
$$
\sum_{i=1}^d 1_{(x_1)_i \neq (x_2)_i} | \sigma(2/\eta)-\sigma(-2/\eta)| =  (  1 - 2 \sigma(-2/\eta) ) \times \ell(x_1,x_2).$$
It is therefore sufficient to prove the contraction of $\hat v(\cdot|\cdot)$.
To that end, we observe that $\hat v(\cdot|\cdot)$ equals the transition kernel $t_{\rm DULA}(\cdot|\cdot)$ by replacing $s(\cdot)/2$ by $s(\cdot)$, which contraction condition is given in \eqref{eq:contraction-dula}: $\hat v(\cdot|\cdot)$ is therefore contractive when $4d\beta_2e^{4\beta_1}\leqslant1$.

    \subsection{Proof of \Cref{t:discretization-DUPS}}\label{app:discretization-DUPS}

    By the Taylor expansions used in the proof of \Cref{t:discretization-dula}, by observing that $u(\cdot|\cdot)$ equals $t_{\rm DULA}(\cdot|\cdot)$ by replacing $s(\cdot)$ by $0$, we have
        \begin{equation*}
        \begin{array}{ll}
           u (z|x) = 1-de^{-2/\eta} + \underset{\eta\to0}{O}\left( e^{-4/\eta} \right)&   \text{if } z=x \\
           u (z|x) = e^{-2/\eta}  + \underset{\eta\to0}{O}\left( e^{-4/\eta} \right)& \text{if } \ell(x,z)=1  \\
            u (z|x) = \underset{\eta\to0}{O}\left( e^{-4/\eta} \right) & \text{otherwise.}
        \end{array}
    \end{equation*}
    Likewise, $\hat v(\cdot|\cdot)$ equals $t_{\rm DULA}(\cdot|\cdot)$ by replacing $s(\cdot)/2$ by $s(\cdot)$, therefore we have
    \begin{equation*}
        \begin{array}{ll}
            \hat v (x'|z) = 1-e^{-2/\eta} \sum_{i=1}^d \exp(-2z_i s(z)_i) + \underset{\eta\to0}{O}\left( e^{-4/\eta} \right)&   \text{if } x'=z \\
            \hat v (x'|z) = e^{-2/\eta} \exp (-2z_i s(z)_i) + \underset{\eta\to0}{O}\left( e^{-4/\eta} \right)& \text{if } x' = [-z_i,z_{-i}]  \\
            \hat v (x'|z) = \underset{\eta\to0}{O}\left( e^{-4/\eta} \right) & \text{otherwise.}
        \end{array}
    \end{equation*}
    Hence, by sum, we have
        \begin{equation*}
        \begin{array}{ll}
            t_{\rm DUPS} (x'|x) = 1-e^{-2/\eta} \sum_{i=1}^d \left( 1+\exp(-2x_i s(x)_i) \right)+ \underset{\eta\to0}{O}\left( e^{-4/\eta} \right)&   \text{if } x'=x \\
            t_{\rm DUPS} (x'|x) = e^{-2/\eta}\left( 1+ \exp (-2x_i s(x)_i) \right) + \underset{\eta\to0}{O}\left( e^{-4/\eta} \right)& \text{if } x' = [-x_i,x_{-i}]  \\
            t_{\rm DUPS} (x'|x) = \underset{\eta\to0}{O}\left( e^{-4/\eta} \right) & \text{otherwise.}
        \end{array}
    \end{equation*}
    This proves that the DUPS is the discretization of $\{\mathbf{X}_t\}_{t\geqslant0}$ with step-size $e^{-2/\eta}$.
    Besides, we have
    $$
    v(x'|z) \propto p(x') \exp\left( \frac{z^\top x'}{\eta} \right) \propto \exp\left( \log p(x')-\log p(x) \right) \exp\left( -2 \frac{\ell(x',z)}{\eta} \right) .
    $$
    Thus, its Taylor expansion at order $e^{-2/\eta}$ when $\eta\to0$ is
    \begin{equation*}
        \begin{array}{ll}
             v (x'|z) = 1-e^{-2/\eta} \sum_{i=1}^d \exp (\log p([-z_i, z_{-i}])-\log p(z)) + \underset{\eta\to0}{O}\left( e^{-4/\eta} \right)&   \text{if } x'=z \\
             v (x'|z) = e^{-2/\eta} \exp   (\log p([-z_i, z_{-i}])-\log p(z)) + \underset{\eta\to0}{O}\left( e^{-4/\eta} \right)& \text{if } x' = [-z_i,z_{-i}]  \\
             v (x'|z) = \underset{\eta\to0}{O}\left( e^{-4/\eta} \right) & \text{otherwise.}
        \end{array}
    \end{equation*}
    This is exactly the Taylor expansion of $\hat v(x'|z)$ at order $e^{-2/\eta}$ when $s(\cdot)=\delta \log p(\cdot)$.
    Therefore in that case, $t_{\rm prox}(\cdot|\cdot)$ is also the discretization of $\{\mathbf{X}_t\}_{t\geqslant0}$ with step-size $e^{-2/\eta}$.

\subsection{Proof of \Cref{t:contraction-DUPS-eta-small}}\label{app:contraction-DUPS-eta-small}

    We refine here the contraction rate of $\hat v(\cdot|\cdot)$ in contrast to \eqref{t:contraction-dups}. The computations are very similar to the proof of \Cref{t:contraction-dula-etasmall}
    and are given here for the sake of completeness.
    Let $x,y\in\X$, we have then :
\begin{align*}
\mathcal W & ( \hat v(\cdot|x), \hat v(\cdot|y) )
  \leqslant \sum_{i=1}^d \left| \sigma\left( \frac{2}{\eta} x_i + 2s(x)_i\right) - \sigma\left( \frac{2}{\eta} y_i + 2s(y)_i\right) \right| \\
&  = \sum_{x_i = y_i} \left| \sigma\left( \frac{2}{\eta} + 2x_i s(x)_i\right) - \sigma\left( \frac{2}{\eta} +2 y_i s(y)_i\right) \right|  \\
& \qquad + \sum_{x_i = -y_i}\left| \sigma\left( \frac{2}{\eta} +2 x_i s(x)_i\right) - \sigma\left( \frac{2}{\eta} x_i y_i +2 x_i s(y)_i\right) \right|
\\
&  \leqslant \frac{2}{2 + e^{1/\eta}} \sum_{x_i = y_i} \left| s(x)_i - s(y)_i \right| 
  + \sum_{x_i = -y_i}\left| 1 - \sigma\left( -\frac{2}{\eta} - 2x_i s(x)_i\right) - \sigma\left( -\frac{2}{\eta} - 2y_i s(y)_i\right) \right|
\\
 & \leqslant \frac{2d \beta_2 \left\| x - y \right\|_1 }{2 + e^{1/\eta}} + \sum_{x_i = -y_i} \left( 1 - \frac{1}{2} e^{ -\frac{2}{\eta} - 2x_i s(x)_i} - \frac{1}{2} e^{ -\frac{2}{\eta} - 2y_i s(y)_i} \right)
\\
 & = \frac{2d \beta_2 \left\| x - y \right\|_1 }{2 + e^{1/\eta}} + \sum_{x_i = -y_i} \left( 1 -  e^{ -\frac{1}{\eta}} \right)
\\
&  = \frac{2d \beta_2 \left\| x - y \right\|_1 }{2 + e^{1/\eta}} +\left( 1 -  e^{ -\frac{1}{\eta}} \right) \times \ell (x,y) \\
& \leqslant \left( 1 - e^{-1/\eta} \left( 1 - 4 \beta_2 d \right)\right) \times \ell (x,y).
\end{align*}
Thus, if $\beta_2 d \leqslant \frac{1}{8}$, we get a contraction factor of $(1 - \frac{1}{2} e^{-1/\eta})$ for $\hat v(\cdot|\cdot)$.

\subsection{Proof of \Cref{t:error-dups-eta-small}}\label{app:error-dups-eta-small}
    We have
    \begin{align*}
        \mathcal W \left( \hat p_{\rm DUPS}(\cdot) , p(\cdot) \right) & =  \mathcal W \left( \hat p_{\rm DUPS}t_{\rm DUPS}(\cdot) ,p(\cdot) t_{\rm prox} \right)  \\
        & \leqslant \mathcal W \left( \hat p_{\rm DUPS}t_{\rm DUPS}(\cdot) , pt_{\rm DUPS}(\cdot) \right) +  \mathcal W \left( pt_{\rm DUPS}(\cdot) ,p t_{\rm prox}(\cdot) \right) .
    \end{align*}
    Furthermore, we have $\mathcal W \left( \hat p_{\rm DUPS}t_{\rm DUPS}(\cdot) , pt_{\rm DUPS}(\cdot) \right)\leqslant (1 - \frac{1}{2} e^{-1/\eta}) \times \mathcal W \left( \hat p_{\rm DUPS}(\cdot) , p(\cdot) \right)$. For the second term, we have
    \begin{align*}
        \mathcal W \left( pt_{\rm DUPS}(\cdot) ,pt_{\rm prox}(\cdot) \right) & \leqslant \sup_{x\in\X} \mathcal W \left( t_{\rm DUPS}(\cdot|x) , t_{\rm prox}(\cdot|x) \right) \\
        & \leqslant \sup_{z\in\X} \mathcal W \left(\hat v(\cdot|z),v(\cdot|z ) \right) .
    \end{align*}
    Let $z\in\X$. We have, given $x'\sim \hat v(\cdot |z)$ and $y'\sim v(\cdot|z)$:
    \begin{align*}
        \mathcal W \left(\hat v(\cdot|z),v(\cdot|z ) \right) & \leqslant \sum_{i=1}^d  \left| \bP[x_i'=-z_i] - \bP[y_i'=-z_i] \right| .
    \end{align*}
    Moreover,
    $$
        \bP[y_i'=-z_i] = \frac{1}{1+\exp\left( \frac 2\eta - 2 z_is(z)_i \right)},
    $$
    and
    $$
        \bP[x_i'=-z_i] = \bP[x_i'=-z_i | x'_{-i}=z_{-i}] \bP[x'_{-i}=z_{-i}] + \bP[x_i'=-z_i | \ell (x',z)\geqslant 2] \bP[\ell (x',z)\geqslant 2] ,
    $$
    with 
    $$
    \bP[x_i'=-z_i | x'_{-i}=z_{-i}] = \frac{1}{1+\exp\left( \frac 2\eta - 2 z_is(z)_i \right)} = \bP[y_i'=-z_i] .
    $$
    Moreover, since $x_i\delta\log p(x)_i\geqslant-1/2\eta$, then $p(x')\leqslant p(z) e^{\ell (x',z)/\eta}$, and we have then
    \begin{multline*}
    \bP[\ell (x',z)\geqslant 2] = \frac{\sum_{x':\ell (x',z)\geqslant 2} p(x') \exp( -2\ell (x',z))/\eta }{\sum_{x'\in\X} p(x') \exp( -2\ell (x',z)/\eta) } \\
    \leqslant \frac{\sum_{x':\ell (x',z)\geqslant 2} p(z) \exp( -\ell (x',z)/\eta) }{p(z) } 
    \leqslant  \sum_{k=2}^d \binom{d}{k} e^{-k/\eta} \leqslant C e^{-2/\eta}/d,
    \end{multline*}
    which gives
    $$
    \bP[x_i'=-z_i] - \bP[y_i'=-z_i]  \leqslant C e^{-2/\eta}/d .
    $$
    Besides, using that $p(x')\leqslant p(z)e^{\ell (x',z)/\eta}$ for all $x'\in\X$:
    \begin{align*}
     \bP[x'_{-i}=z_{-i}] &  \geqslant \bP [x' = z] \\
     & =  \frac{p(z)}{p(z) + \sum_{x'\neq z} e^{-2\ell (x',z)/\eta} p(x')} \\
     & \geqslant \frac{ p(z)}{ p(z) + \sum_{x'\neq z} e^{-\ell (x',z)/\eta} p(z)} \\
     & \geqslant 1 - \sum_{x'\neq z} e^{-\ell (x',z)/\eta}  \geqslant 1 - \sum_{k=1}^d \binom{d}{k} e^{-k/\eta}  \geqslant 1- C e^{-1/\eta} /d .
    \end{align*}
    This gives
    $$
     \bP[y_i'=-z_i] - \bP[x_i'=-z_i] \leqslant \frac{C e^{-1/\eta}}{d} \frac{1}{1+\exp\left( \frac 2\eta - 2 z_is(z)_i \right)} \leqslant Ce^{-2/\eta}/d. 
    $$
    All in all, we get
    $$
         \mathcal W \left(\hat v(\cdot|z),v(\cdot|z ) \right)  \leqslant Ce^{-2/\eta} 
    $$
    and
    $$
        \mathcal W \left( \hat p_{\rm DUPS}(\cdot) , p(\cdot) \right) \leqslant C e^{-1/\eta} .
    $$

    \section{Convergence of  DMAPS}\label{app:DMAPS}

We want to apply the results in \Cref{sec:MH} to the DULA and the DUPS transition kernels. These algorithms, adjusted by a Metropolis acceptance step \eqref{eq:Azxx'} are in this section as the DMALA (Discrete Metropolis-Adjusted Langevin Algorithm, \cite{zhang2022langevinlike}) and the DMAPS (Discrete Metropolis-Adjusted Proximal Sampler).
We focus here to the analysis of the DMAPS.

We introduce a step-size variable $\eta\in (0,+\infty)$, and we set:
$$
u(z|x) = \frac{1}{(2\cosh(1/\eta))^d} \exp\left( \frac{1}{\eta} x^\top z \right) 
$$
$$
\hat v(x'|z) = \frac{1}{2^d\prod_i \cosh(z_i/\eta+s(z)_i)} \exp\left(  x'^\top \left(\frac{z}{\eta} + s(z)\right) \right) ,
$$
where $s\colon \X \to \mathbb R^d$ is a known score function.
Then, the acceptance rate write as
$$
A_z(x'|x) = \min\left\{ 1, \frac{p(x')}{p(x)} \exp\left(  (x-x')^\top s(z) \right) \right\} .
$$
We explain how the assumptions for \Cref{p:contraction-generic} are satisfied.
\begin{proposition}
Consider the constants $\beta_1,\beta_2\geqslant0$ given by \eqref{eq:beta1} and \eqref{eq:beta2}. 
Suppose that $\log p(\cdot)$ is $\beta_1$-Lipschitz\footnote{This is true when $s(\cdot)$ equals $\nabla\log p(\cdot)$ by extending $\log p(\cdot)$, or when $s(\cdot)$ equals $\delta \log p(\cdot)$. Besides, when $\delta\log p(\cdot)$ satisfies \eqref{eq:beta1} and \eqref{eq:beta2}, then, so does the Gibbs score \eqref{eq:gibbs-score}.}.
Suppose that the transition kernel $t(\cdot|\cdot)$ contracts.
    Then, Assumption~\Cref{A:Lipschitz} holds with $L=6\beta_1+4d^{3/2} (\sigma(-2/\eta) + \sigma(-2/\eta+\beta_1))^{1/2}\beta_2$.
\end{proposition}
\begin{proof}
    Let $x,y\in\X$, $z_1\sim u(\cdot|x)$, $z_2\sim u(\cdot|y)$, $x'\sim \hat v(\cdot|z_1)$ and $y'\sim \hat v(\cdot|z_2)$.
    We write $A_z(x'|x)=\min\{1,\exp(\varphi_z(x,x'))\}$ where
    $$
        \varphi_z(x,x') = \log p(x') - \log p(x) + (x-x')^\top s(z) .
    $$
    Since $\min\{1,\exp(\cdot)\}$ is $1$-Lipschitz, it is sufficient to prove that $x\mapsto \bE \varphi_z(x,x')$ is $L$-Lipschitz. 
    However,
    \begin{multline*}
        \bE\left| \varphi_{z_1}(x,x') - \varphi_{z_2}(y,y') \right|
         \leqslant
        \bE|\log p (x') - \log p (y')|
        + |\log p(x) - \log p(y)| \\
        + \bE | s(z_1)^\top (x-x') + s(z_2)^\top (y'-y) | .
    \end{multline*}
    Yet, we have
    $\bE|\log p (x') - \log p (y')|\leqslant \beta_1\times\bE[\ell(x',y')] \leqslant \beta_1\times \ell(x,y)$ by contraction of $t(\cdot|\cdot)$, and
    $|\log p(x) - \log p(y)| \leqslant \beta_1\times \ell(x,y)$.

    Moreover,
    \begin{multline*}
     \bE | s(z_1)^\top (x-x') + s(z_2)^\top (y'-y) |
     \leqslant
      \bE | (s(z_1)-s( z_2))^\top (x-x') | +  \bE | s( z_2)^\top (y'-x'+x-y) |  \\
      \leqslant  4 \beta_2 \bE[\ell(x,x')\ell(z_1, z_2)] + 2\beta_1 (\bE[\ell(x',y')]+\ell(x,y)) .
    \end{multline*}
   However, we have
    \begin{multline*}
    \bE[\ell(x,x')\ell(z_1, z_2)] \leqslant \bE[\sqrt{ \bE[\ell (x,x')^2] } \sqrt{ \ell (z, z_2)^2}] \\
    \leqslant d^{3/2} (\sigma(-2/\eta) + \sigma(-2/\eta+\beta_1))^{1/2} \times \bE[\ell (z_1, z_2)].
    \end{multline*}
    Since the kernel $u(\cdot|\cdot)$ is contractant (by contraction of the DULA when $s(\cdot)=0$, see \eqref{eq:contraction-dula}), we have $\bE[\ell(z_1, z_2)]\leqslant \ell(x,y)$. By contraction of $t(\cdot|\cdot)$, $\bE[\ell(x',y')]\leqslant \ell(x,y)$. All in all, we get the desired result.
\end{proof}

\begin{proposition}Consider the constants $\beta_1,\beta_2\geqslant0$ given by \eqref{eq:beta1} and \eqref{eq:beta2} and suppose that the score $s(\cdot)$ equals $\nabla \log p(\cdot)$ when extending $\log p(\cdot)$ to $\mathbb R^d$.
        Then, assumption \Cref{A:lower-bound} holds with $$\delta = 1 -\exp\left(  - 2 \beta_2 d^2 \times \left(  \sigma(2/\eta) + \sigma(2/\eta -\beta_1) \right) + 2 \left( \sigma(2/\eta)^2 + \sigma(2/\eta)\sigma(2/\eta-\beta_1) \right) ^{1/2}  \right).$$
    \end{proposition}
    \begin{proof}
        Since $s=\nabla \log p$, and as a consequence of \eqref{eq:beta2}, we have
        $$
        p(x') \geqslant p(x) \exp\left( s(x)^\top (x'-x) -  \frac{\beta_2}{2} \|x-x'\|_1^2\right) .
        $$
        Therefore
        $$
        A_z(x'|x) \geqslant 
        \exp\left(  - \frac{\beta_2}{2} \| x - x'\|_1^2 - \beta_2 \| x - x'\|_1 \| x - z\|_1 \right) .
        $$
        Using the Cauchy-Schwarz inequality, we get:
        $$
            \bE[A_z(x'|x)] \geqslant \exp\left(  -2\beta_2\bE\ell(x,x')^2 -4\beta_2 \sqrt{ \bE\ell( x , x')^2 \bE\ell (x , z)^2 }  \right) .
        $$
        Yet, $\bE\ell (x,z)^2 \leqslant d\bE\ell (x,z) \leqslant d^2 \sigma(-2/\eta)$ and $\bE\ell(x,x')^2\leqslant d \bE \ell (x,x') \leqslant d^2 \times (\sigma(-2/\eta) + \sigma(-2/\eta+\beta_1) )$. Therefore,
        $$
        \bE[A_z(x'|x)] \geqslant 
        \exp\left(  - 2 \beta_2 d^2 \times \left(  \sigma(2/\eta) + \sigma(2/\eta -\beta_1) \right) + 2 \sqrt{\sigma(2/\eta)^2 + \sigma(2/\eta)\sigma(2/\eta-\beta_1)}   \right) .
        $$
        \end{proof}
All in all,     
we obtain the following contraction rate of the DMAPS.

\begin{theorem}
[Contraction rate of DMAPS]
Consider the constants $\beta_1,\beta_2\geqslant0$ given by \eqref{eq:beta1} and \eqref{eq:beta2} and suppose that the score $s(\cdot)$ equals $\nabla \log p(\cdot)$ when extending $\log p(\cdot)$ to $\mathbb R^d$. Suppose that $t_{\rm DMAPS}(\cdot|\cdot)$ contracts with rate $1-\varepsilon$.
Then we have
        \begin{multline}
            \mathcal W(t_{\rm DMAPS}(\cdot|x), t_{\rm DMAPS}(\cdot|y)) \\ \leqslant \left[ 2- \varepsilon
             -e^{  - 2 \beta_2 d^2 \times \left(  \sigma(2/\eta) + \sigma(2/\eta -\beta_1) \right) + 2 \left( \sigma(2/\eta)^2 + \sigma(2/\eta)\sigma(2/\eta-\beta_1) \right) ^{1/2}}  \right. \\
             \left.+ 6d\beta_1+4d^{5/2} (\sigma(-2/\eta) + \sigma(-2/\eta+\beta_1))^{1/2}\beta_2 \right] \times \ell(x,y) .
        \end{multline}
\end{theorem}

\end{document}